\documentclass[pdflatex,sn-mathphys-ay]{sn-jnl}

\usepackage[all]{xy}

\usepackage{graphicx}%
\usepackage{multirow}%
\usepackage{amsmath,amssymb,amsfonts}%
\usepackage{amsthm}%
\usepackage{mathrsfs}%
\usepackage[title]{appendix}%
\usepackage{xcolor}%
\usepackage{textcomp}%
\usepackage{manyfoot}%
\usepackage{booktabs}%
\usepackage{algorithm}%
\usepackage{algorithmicx}%
\usepackage{algpseudocode}%
\usepackage{listings}%

\theoremstyle{thmstyleone}%
\newtheorem{theorem}{Theorem}
\newtheorem{proposition}[theorem]{Proposition}%

\theoremstyle{thmstyletwo}%
\newtheorem{example}{Example}%
\newtheorem{remark}{Remark}%
\newtheorem{lemma}{Lemma}%
\theoremstyle{thmstylethree}%
\newtheorem{definition}{Definition}%
\newtheorem{corollary}{Corollary}
\begin{document}
\title[On examples of duals Saito's basis of some inhomogeneous divisors, and application.]{On examples of duals Saito's basis of some inhomogeneous divisors, and application.}

\author*[1]{\fnm{Kari} \sur{Kamtila}}\email{karikamtila@gmail.com}\equalcont{These authors contributed equally to this work.}

\author[1]{\fnm{Joseph} \sur{Dongho}}\email{josephdongho@yahoo.fr}\equalcont{These authors contributed equally to this work.}

\author[2]{\fnm{Prosper Rosaire } \sur{Mama Assandje}}\email{mamarosaire@facsciences-uy1.cm}\equalcont{These authors contributed equally to this work.}

\author[2]{\fnm{Thomas } \sur{Bouetou Bouetou}}\email{tbouetou@gmail.com}
\equalcont{These authors contributed equally to this work.}

\affil*[1]{\orgdiv{Faculty of science}, \orgname{University of
Maroua}, \city{Maroua}, \postcode{814 Maroua}, \country{Far-North,
Cameroon}}

\affil[2]{\orgdiv{Faculty of science}, \orgname{University of
Yaounde 1}, \city{Yaounde}, \postcode{337 Yaounde 1},
\country{Center,Cameroon}}


\abstract{We investigate a class of  non-quasi-homogeneous free
divisors in the sense of Saito.  These divisors are defined by
equations of the form $D:= \{h=0\}$  on $\mathbb{C}^p$, where the
polynomial $h$ is specific linear combination of monomials involving
the product of coordinates. For this class, we explicitly construct
a Saito basis for the module of logarithmic vector fields
$Der(logD)$. This construction is then applied to the setting of
logarithmic Poisson geometry. Focusing  on the example defined by
$h=xy+x^{2}y^{2}+x^3y^3$ on the Poisson algebra
$(\mathcal{A}=\mathbb{C}[x,y], \{-,-\}_{h})$, where the Poisson
bracket is induced by the bivector $\pi = h\partial x\wedge\partial
y$. We define the associated Koszul bracket on the module of
logarithmic 1-forms. This enables us to prove that $\pi$ endows the
sheaf of logarithmic 1-forms $\Omega^{1}(log D )$ with a
Lie-Rinehart algebra structure. Furthermore, we introduce and
provide explicit descriptions for the resulting cohomology theory,
which we term the logarithmic Poisson cohomology $H_{log}^{\bullet}
$ of  $\{-,-\}_{h}$. As a related and foundational computation, we
also calculate the corresponding logarithmic De Rham cohomology
$H^{\bullet}_{DR}$ for the divisor $D$ and we make a generalization
in dimension 2.}



\maketitle
\section{Introduction}\label{sec1}

The free divisors were introduced by  \cite{KS}. He established the
basics properties in low dimensions and introduced the corresponding
sheaves of logarithmic forms and vector fields, and prove that they
are reflexives in dimension 2. This paper specially study the case
of free divisors that are not quasi-homogeneous. We investigate
their algebraic structures and compute their associated logarithmic
cohomology. The motivation for this work is that of constructing the
Saito basis for the modules of the logarithmic 1-forms and the
logarithmic  vector fields along the free non quasi-homogeneous
divisor $D=\{h=0\}$. We also  examine their associated cohomological
invariants, in order to make a comparison. Although \cite{KS} has
prescribed construction methods, but their implementations are not
always obvious, especially for these types of divisors. Several
researchers have alluded almost free weighted divisors in their
work. Particularly \cite{FJCM}  gave a formula for the logarithmic
de Rham complex with respect to a free divisor in terms of
$\mathcal{V}_{0}$ -modules,  generalizing the classical formula for
de Rham complex in terms of $\mathcal{D}$-modules. On the same line,
\cite{FJCJN}  used this framework to provide algorithms for
computing basis of the logarithmic cohomology groups for square-free
polynomials in dimension 2. \cite{JS} later provided classifications
and basis for weighted homogeneous cases. \cite{VK} studied the
theory of Poisson algebras, developed by \cite{AL} an geometrically
studied by \cite{JB}, led to the crucial notion of Poisson
cohomology while computations have been achieved for
quasi-homogeneous cases by researchers such as
  \cite{PA}, the systematic work of logarithmic Poisson cohomology for some homogeneous divisors, initiated by
 \cite{DJ1},   remain an open and challenging area.
 Unlike  De Rham cohomology, the Poisson cohomology group of given Poisson manifold are largely
  irrelevent to the manifold's  to the topology  and moreover  have poor functorial properties.
     They are very large, and their explicit computation is considerably  more    complicated.
     However, they are of great interest because they encode essential information about Poisson
     structures.   The algebraic aspects of this theory were developed by   \cite{JH}
        and  the geometric setting by   \cite{IV}.
   \cite{CRPV}, \cite{PM} and  \cite{PA} have computed the Poisson
 cohomology for the case where the divisors are
 quasi-homogeneous.
  According to \cite{JH}, when the Poisson structure $\{-,-\}$ comes from a symplectic structure
   $\sigma$ on a smooth manifold $N$, it  induces an
    isomorphism $\sigma^{*}:~~ H^{*}_{DeRham}(N,R) \longrightarrow H^{*}_{Poisson}(N, \{-,-\} , R)$
    where $R$ is a commutative ring.
 The central research questions addressed here are:
 what are  the Saito basis for the modules of the logarithmic 1-forms and
  the logarithmic  vector fields along the free non quasi-homogeneous divisor $D=\{h=0\}$ and  the
   logarithmic cohomology associated to a specific case of inhomogeneous free divisors $D$ where
$h=xy+x^{2}y^{2}+x^3y^3$ and $ h=\overset{m}{\underset{i=1}{\sum}}
   \alpha_{i} x_{1}^{n_{i}} x_{2}^{n_{i}}$?
   What happens with this isomorphism when the Poisson structure is log-symplectic as our case?
The expected results are three folds. Firstly, we determine Saito
basis for the module of logarithmic vector fields along $D$ noted
$Der_{X,x}(logD) \simeq \mathcal{O}_{X,x}(\delta_{n} +
\delta_{n}^{\overline{q}}) \oplus \mathcal{M}$,  for $D=\{h=
\overset{n}{\underset{i=1}{\sum}}\alpha_{i}(\overset{p}{\underset{j=1,}{\Pi}}
x_{j})^{n_{i}}=0\}$ with $D\subset X \subset \mathbb{C}^{p}$. where
 the submodule  $\mathcal{M}$ is given by $\mathcal{M} = \{ \delta^{0} \in Der_{X,x}(logD)~ / ~~\delta^{0}(h) = 0\}$;
 $\delta_{n} =
\overset{n}{\underset{i=1}{\sum}}
\alpha_{i}(\overset{p}{\underset{j=1,}{\Pi}} x_{j})^{n_{i}-1} E_{p}$.  $E_{p}= \overset{p}{\underset{i=1}{\sum}} x_{i}\partial _{x_{i}}$ is the local Euler vector field on $X$;
and  $\delta_{n}^{\overline{q}}=
 \overset{q_{1}}{\underset{i=1}{\sum}}
 \alpha_{i}(\overset{p}{\underset{j=1,}{\Pi}} x_{j})^{n_{i}-1}
 x_{1}\partial _{x_{1}}  +
 \overset{q_{2}}{\underset{i=q_{1}+1}{\sum}}
 \alpha_{i}(\overset{p}{\underset{j=1,}{\Pi}} x_{j})^{n_{i}-1}
 x_{2}\partial _{x_{2}}  +  ... +
 \overset{q_{n}}{\underset{i=q_{n-1}}{\sum}}
 \alpha_{i}(\overset{p}{\underset{j=1,}{\Pi}} x_{j})^{n_{i}-1}
 x_{n}\partial _{x_{n}}$.
Secondly, we compute the associated Koszul bracket. Thirdly, we
provide explicit computations of the  logarithmic Poisson cohomology
along an ideal $h\mathcal{A}$ and we get
  $H_{log}^{0}  \simeq   \mathbb{C}$,
 $H_{log}^{1}  \simeq \mathbb{C}\times (\mathbb{C} \oplus xy \mathbb{C})$,
  $H_{log}^{2}  \simeq   \mathbb{C} \oplus xy \mathbb{C}$;
 and we also compute the logarithmic De Rham cohomology along the divisor $D=\{h=0\}$. We show that,
  if $\mathcal{A}= \mathbb{C}[x,y]$
  we get
   $H^{0}_{DR}\simeq   \mathbb{C}$,
  $H^{1}_{DR} \simeq  0 \times \mathbb{C}$ ,
  $H^{2}_{DR}  \simeq H_{log}^{2}$, moreover
       if $\mathcal{A}= \mathbb{C}[x_{1},x_{2}]$ we generalize the results in dimension 2 with  $ h=\overset{m}{\underset{i=1}{\sum}}
     \alpha_{i} x_{1}^{n_{i}} x_{2}^{n_{i}}$ and $\tilde{h}_{f}\in \mathcal{F}$(where submodule $\mathcal{F}$ is defined by relation $(\ref{13})$ in section 5) as
    $H_{log}^{1}  \simeq  \mathbb{C} \times  \mathbb{C}$ and $H_{log}^{2}  \simeq   \mathbb{C}  \simeq H^{2}_{DR}$, if $d^{\circ}(\tilde{h}_{f})_{x_{i}}= 0$ and
     $H_{log}^{1}  \simeq  \mathbb{C} \times \left( \mathbb{C}\oplus \tilde{h}_{f}\mathbb{C} \right)$ and $H_{log}^{2}  \simeq   \mathbb{C}\oplus \tilde{h}_{f}\mathbb{C}   \simeq H^{2}_{DR}$, if $d^{\circ}(\tilde{h}_{f})_{x_{i}} > 0$.
This paper is structured as follows. After this introduction, the
section 2 covers the necessary preliminaries on free divisors,
Poisson structure, the Weyl algebra and the cochain complex. The
section 3 is dedicated to constructing the Saito basis of an
inhomogeneous divisors class. In section 4, we determine the induced
Koszul bracket on the module of logarithmic 1-forms. We complete in
section 5  whit computations of the logarithmic Poisson and De Rham
cohomology groups.

\section{Preliminaries and main results}\label{Rapp}

\subsection{On Poisson and Weyl  algebras}
\begin{definition} \cite{DJ1}
A Poisson's bracket on $\mathcal{A}$  is given by
a bilinear mapping  $\{-,-\}:\mathcal{A}\times\mathcal{A}\rightarrow\mathcal{A}$, satisfying the following conditions:
  \begin{enumerate}
 \item[1.] $\{a,b\}=-\{b,a\}$, antisymmetry;
 \item[2.] $\{a,\{b,c\}\}+\{b,\{c,a\}\}+\{c,\{a,b\}\}=0$, Jacobi identity;
 \item[3.] $\{a,bc\}=b\{a,c\}+c\{a,b\}$, Leibnitz rule.
 \end{enumerate}
 The pair $(\mathcal{A}, \{-, -\})$ is what we call the Poisson algebra.
\end{definition}

\begin{example}
let  $h=xyz+x^2y^2z^2+x^3y^3z^3\in \mathcal{A}=\mathbb{C}[x,y,z]$. We have the respective  Poisson structures
$\{x, y\}_{h} =  xy  \tilde{h}$,
$\{ y,z\}_{h} = yz \tilde{h}$ and
$\{z, x\}_{h} = xz \tilde{h}$  with $\tilde{h} = 1+2xyz + 3 x^2 y^2 z^2$. The Poisson's bracket associate to $h$ is
$\{ a,b\}_{h} =  \tilde{h} (xy  \partial _{x} a \wedge \partial _{y} b + yz \partial _{y} a \wedge \partial _{z} b + zx \partial _{z} a \wedge \partial _{x} b)$ for all $a,b \in \mathcal{A} = \mathbb{C}[x,y,z]$.
We also can  define it for $h \in \mathcal{A} =  \mathbb{C}[x_{1},x_{2}]$
 by the following $\{f, g\}_{h} =  h  \left(\dfrac{\partial f}{\partial x_{1} } \dfrac{\partial g}{\partial x_{2} } - \dfrac{\partial f}{\partial x_{2} } \dfrac{\partial g}{\partial x_{1} }\right)$ for  $f,g \in \mathbb{C}[x_{1},x_{2}]$. And  $(\mathcal{A}, \{-, -\}_{h})$ is the Poisson algebra.
\end{example}

 \begin{definition} \cite{DJ1}
The Weyl algebra of order $p$ over $\mathbb{K}$(a field of
characteristic 0) is the nuciferous associative algebra on
$\mathbb{K}$ which has $2p$ generators $x_{1},...,x_{p};
\partial_{1},...,\partial _{p}$ noted by $A_{p}(\mathbb{K}) =
(x_{1},...,x_{p}; \partial_{1},...,\partial _{p})$, checking the
following relationships:
\begin{enumerate}
 \item[1.]  $[x_{i},x_{j}] = [\partial _{i},\partial_{j}]= 0$ for $i \neq j $;
 \item[2.] $[x_{i} , \partial_{j}] = \delta_{ij}$, where $\delta_{ij}$ is the
  Kronecker symbol and $\partial _{i}=\partial x_{i}$.
\end{enumerate}
It is also noted by $A_{p}$ if there is no ambiguity. Particulary $A_{0}(\mathbb{K})=A_{0} = \mathbb{K}$.
\end{definition}

 \subsection{On  Kyoji Saito's basis of the free divisors}

  Let $X$ be the
 $p$-dimensional affine space on $\mathbb{C}^{p}$ and
 $x=(x_{1},...,x_{p})$ be its coordinate system.
 \begin{definition} \cite{PA}
  A polynomial $h\in \mathbb{C}[x_{1},...,x_{p}]$ is said to be quasi-homogeneous of degree $\varpi$,
  if there exist the vector field
  $\delta = \overset{p}{\underset{i=1}{\sum}}\varpi_{i}x_{i}\partial_{x_{i}}$ such that $\delta(h)=\varpi h$
   where each $\varpi_{i}\in \mathbb{N}^{*}$ is the weight of $x_{i}$. It is equivalent to say that
   $h(t^{\varpi_{1}} x_{1},...,t^{\varpi_{p}} x_{p}) = t^{\varpi} h(x_{1},...,x_{p})$.
   \end{definition}

   \begin{definition} \cite{PA}
     A divisor $D=\{h=0\}$  is said to be quasi-homogeneous(or weighted homogeneous) of degree $\varpi$ at $x\in X$,
      if $h$ is  weighted homogeneous of degree $\varpi$ at $x\in X$.
      \end{definition}
      For example, if $h=x^3+y^{2}$ we have $\delta=2x\partial x + 3y \partial y$, $\delta(h)=6h$.
       Then polynomial  $h$ is weighted homogeneous of degree $\varpi = 6$ and  $D=\{h=x^3+y^{2}=0\}$
        is the weighted homogeneous divisor.

 \begin{definition} \cite{KS}
  Let $U$ be a domain in $\mathbb{C}^{p}$ and let $D\subseteq U$ be a divisor of $\mathbb{C}^{p}$, defined by
   a reduced equation $h=0$, where $h$ is a holomorphic function on $U$. A mesomorphic $q$-form $\omega$ on $U$
   is called a logarithmic $q$-form (along $D$) if $h\omega$ and $hd\omega$ are holomorphic on $U$. Since most
    of the time the divisor $D$ is fixed, we will simply speak of logarithmic $q$-forms.
 \end{definition}

\begin{remark}\cite{KS}
 A mesomorphic $q$-form is logarithmic(along $D$) at a point $x$ if $\omega h_{x}$ and $h_{x}d\omega$
  are holomorphic in open neighborhood around $x$ and denote by the following:
 $\Omega_{X,x}^{q}(logD)=\{\omega: ~\omega~ germ ~of~ a~ logarithmic ~ q-form~ at~ x \}$.
\end{remark}

 \begin{theorem} \cite{KS} \label{T1}
  \begin{enumerate}
  \item[i.] $\Omega_{X,x}^{1}(logD)$ is $\mathcal{O}_{X,x}$-free if and only if  $\Omega_{X,x}^{1}(logD) = \wedge^{p}\Omega_{X,x}^{1}(logD)=\Omega_{X,x}^{p}(logD)$. i.e if there exist $p$ elements $\omega_{1},..., \omega_{p}$ of $\Omega_{X,x}^{1}(logD)$ such that $\omega_{1}\wedge...\wedge \omega_{p} = \dfrac{unit}{h}dx_{1}\wedge...\wedge dx_{p}$. Then a set of forms $\{\omega_{1},..., \omega_{p}\}$ make a system of $\mathcal{O}_{X,x}$-free basis for $\Omega_{X,x}^{1}(logD)$.
 \item[ii.] $Der_{X,x}(logD)$ is $\mathcal{O}_{X,x}$-free if and only if there exist $p$ elements $\delta^{1},...,\delta^{p}\in Der_{X,x}$ with $\delta^{i} = \overset{p}{\underset{i =1}{\sum}}a_{i}^{j}(z)\partial_{z_{j}}$, $j=1,...,p$ such that the determinant $det(a_{i}^{j}(z))_{_{i,j=1,...,p}}$ is a unit multiple of $h_{x}$. Then the set of the vector fields $\{\delta^{1},...,\delta^{p}\}$ is free basis of $Der_{X,x}(logD)$.
  \end{enumerate}
 \end{theorem}
 \begin{proposition}\cite{KS} $\label{P2}$
      Let $\delta$ a vector field on $X$. The following properties are equivalent:
      \begin{enumerate}
      \item[i.]
     For every smooth point $x$ of $D$, the  tangent vector $\delta(x)$ at $x$ is tangent to $D$,
      \item[ii.] For every  point $x$ of $D$, if $h_{x}$ is the defining function of $D$, then  $\delta h_{x}$
       is in the ideal $(h_{x})\mathcal{O}_{X,x}$.
      \end{enumerate}
      \end{proposition}

  \begin{proposition}\cite{EF} $\label{P3}$
   Suppose that $\delta^{1},...,\delta^{p}$ form a basis of $ Der_{X,x}(log D)$.
   Then $[\delta^{i},\delta^{j}] = 0$ for all $i,j\in \{1,...,p\}$ if and only
   the basis $\omega_{1},...,\omega_{p}$ of $ \Omega_{X,x}(log D)$
   satisfying $\delta^{i}\lrcorner \omega _{j}=\delta_{ij}$   consists of closed forms.
  \end{proposition}

   \begin{definition} \cite{DJ1}
  A  vector field $\delta$ is logarithmic along $D$ or simply logarithmic, if
   the equivalent conditions of \textbf{proposition} $\ref{P2}$  hold.
        \end{definition}
 \subsection{On the logarithmic cochain complexes}
 \begin{definition} \cite{AJ}
 An algebraic cochain complex consists of a sequence of
 $\mathcal{A}$-modules $\textbf{ C}^{*}$ and $\mathcal{A}$-module
 homomorphisms $d^{q}: C^{q} \longrightarrow C^{q+1}$ such that
 $d^{q+1}\circ d^{q}=0$, $q\in \mathbb{Z}$. It is denoted by
 $(\textbf{ C}^{*}, d^{*})$ and it induce the following sequence:
 \[\xymatrix{(\textbf{ C}^{*}, d^{*}):  ...\ar[r]^{d^{q-2}}& C^{q-1} \ar[r]^{d^{q-1}}& C^{q}  \ar[r]^{d^{q}}& C^{q+1}   \ar[r]^{d^{q+2}}& ...}\]
 \end{definition}
 $d:=d^{q}$ is the coboundary operator. Elements of $(C^{q}, d^{q})$
 are called $q$-cochain, those of $Z^{q}(\textbf{ C}^{*}, d^{*})=Ker
 d^{q}$ are called $q$-cocycles and those of $B^{q}(\textbf{ C}^{*},
 d^{*})= Imd^{q}$ are called $q$-coboundary. Since $d^{2}=0$, then
 $B^{q}(\textbf{ C}^{*},  d^{*})$ is submodule of  $Z^{q}(\textbf{
 C}^{*}, d^{*})$ and $H^{q}(\textbf{ C}^{*}, d^{*})=
 \dfrac{Z^{q}(\textbf{ C}^{*}, d^{*})}{B^{q}(\textbf{ C}^{*},
 d^{*})}$ is well defined and called the $q-th$ cohomology
 $\mathcal{A}$-module of $(\textbf{ C}^{*}, d^{*})$.
 \begin{definition} \cite{AJ}
 A cochain morphism $f^{*}: (\textbf{ C}^{*}, d^{*}) \longrightarrow
 (\tilde{\textbf{ C}^{*}}, \tilde{d}^{*})$ consists of a sequence of
 the homomorphisms of $\mathcal{A}$-module $f^{q}: (C^{q}, d^{q})
 \longrightarrow (\tilde{C}^{q}, \tilde{d}^{q})$ that commute with
 the coboundary operators. That is equivalent to say that the squares
 of the following diagram:
  \[ \xymatrix{ ...\ar[r]^{d^{q-1}}& C^{q-1} \ar[d]^{f^{q-1}}\ar[r] ^{d^{q}} & C^{q}\ar[d]^{f^{q}}\ar[r]^{d^{q+1}}&C^{q+1} \ar[d]^{f^{q+1}}\ar[r] ^{d^3}& ...\\
       ...\ar[r]^{\tilde{d}^{q-1}} & \tilde{C}^{q-1} \ar[r] ^{\tilde{d}^{q}}& \tilde{C}^{q}  \ar[r]^{\tilde{d}^{q+1}} &  \tilde{C}^{q+1}  \ar[r]& ...}
      \]
 \end{definition}
      \begin{remark} \cite{DJ1} and \cite{DJ2} \label{R2}\\
      $\mathcal{L}alt( \Omega_{\mathcal{A}}^{1}(log\mathcal{I}),
       \mathcal{A}) = \underset{k\geq
       0}{\bigoplus}\mathcal{L}^{k}alt(\Omega_{\mathcal{A}}^{1}(log\mathcal{I}),
       \mathcal{A}) $  the $\mathcal{A}$-module of the $k$- multilinear
       alternates forms on $\Omega_{\mathcal{A}}^{1}(log\mathcal{I})$ and $\mathcal{I}$ be an ideal of $\mathcal{A} = \mathbb{C}[x_{1},...,x_{p}]$.
      The differential $d^{k}_{\tilde{H}}$ given by
        $d^{k}_{\tilde{H}}:  \mathcal{L}alt( \Omega_{\mathcal{A}}^{1}(log\mathcal{I}), \mathcal{A})
        \longrightarrow \mathcal{L}alt( \Omega_{\mathcal{A}}^{1}(log\mathcal{I}), \mathcal{A})$ is the logarithmic Poisson differential such that,
        for all  $f \in \mathfrak{L}_{alt}(\Omega_{\mathcal{A}}^{1}(log\mathcal{I}), \mathcal{A})$
        and $ x_{1},...,x_{p}\in \Omega_{\mathcal{A}}^{1}(log\mathcal{I})$:
        \begin{eqnarray*}
          d^{i}_{\tilde{H}}f(x_{1},...,x_{p})
           &=& \sum_{i=1}^{p}(-1)^{i-1}\tilde{H}(x_{i})f( x_{1},...,\hat{x}_{i},...,x_{p})+ \label{*} ~~~~~~~~~~~~~~~~~~  \\
           & &\sum_{1\leq i\leq j \leq p} (-1)^{i+j-1}f([x_{i}, x_{j}]_{\Omega_{\mathcal{A}}(log\mathcal{I})},
           x_{1},...,\hat{x}_{i},
           ...,\hat{x}_{j},...,x_{p}).
        \end{eqnarray*}
The cochain complex associated to the above differential is given by the following:
      \begin{displaymath} \label{CC1}
         ...\stackrel{d^{i}_{\tilde{H}}}{\longrightarrow}\Omega^{i}_{\mathcal{A}}(log \mathcal{I})
            \stackrel{d^{i+1}_{\tilde{H}}}{\longrightarrow} \Omega^{i+1}_{\mathcal{A}}(log \mathcal{I}) \stackrel{d^{i+2}_{\tilde{H}}}{\longrightarrow} ...
            \end{displaymath}
  The corresponding cohomology is the Lichnerowicz-Poisson cohomology  called logarithmic Poisson cohomology of $(\mathcal{A}, [-,-], \tilde{H})$
          where $\tilde{H}: \Omega_{\mathcal{A}}(log \mathcal{I}) \longrightarrow Der_{\mathcal{A}}(log \mathcal{I})$ is the logarithmic hamiltonian map.
      \end{remark}

\textbf{Some notations:}\\
In the following, we note that $\mathcal{D}_{X}$ is  the sheaf of
 differential operators on the algebraic manifold $X=\mathbb{C}^{p}$, $\mathcal{V}_{0}^{D}(\mathcal{D}_{X})$ is the ring of logarithmic differential operator along the divisor $D=\{h=0\}$.
Let  $\widetilde{Der _{\mathcal{A}}}(logh)=\{     \delta+\dfrac{\delta(h)}{h} | \delta \in Der _{\mathcal{A}}(log h) \}$.  For the Spencer algebra  $A_{p}$, we  have the quotient $A_{p}$-modules noted by
  $ M^{logh}= \dfrac{A_{p}}{A_{p}Der_{ \mathcal{A}}(log h)}$ and
$ \widetilde{M}^{logh}= \dfrac{A_{p}}{A_{p}\widetilde{Der _{\mathcal{A}}(log h)}}$.\\
If $D$ is free, then $ M^{logh}\simeq M^{logD}$, $ \widetilde{M}^{logh}\simeq \widetilde{M}^{{logD}}$ and  $\widetilde{Der _{\mathcal{A}}}(logh)\simeq \widetilde{Der _{\mathcal{A}}}(log D)$.

\begin{definition} \cite{FJCM}
We call the logarithmic Spencer complex, and denote by $\mathcal{S}p^{\bullet}(log D)$, the complex $\left(\mathcal{V}_{0}^{D}(\mathcal{D}_{X}) \otimes_{\mathcal{O}_{X}} \wedge ^{*}Der(log D) ,~\epsilon_{-*}\right)$  given by:\\
$0\longrightarrow \mathcal{V}_{0}^{D}(\mathcal{D}_{X}) \otimes_{\mathcal{O}_{X}} \wedge ^{n}Der(log D) \longrightarrow  ... \longrightarrow \mathcal{V}_{0}^{D}(\mathcal{D}_{X}) \otimes_{\mathcal{O}_{X}} \wedge ^{1}Der(log D)  \longrightarrow \mathcal{V}_{0}^{D}(\mathcal{D}_{X}) $
where $\epsilon_{-p}(P\otimes(\delta^{1}\wedge...\wedge\delta^{p}))= \overset{p}{\underset{i=1}{\sum}}(-1)^{i-1}P\delta^{i}\otimes (\delta^{1}\wedge...\wedge\hat{\delta^{i}} \wedge ... \wedge \delta^{p}) +$\\
$ \overset{}{\underset{1 \leq i<j\leq p}{\sum}}(-1)^{i+j}P\otimes([\delta^{i},\delta^{j}]\wedge\delta^{1}\wedge ...\wedge \hat{\delta}_{i}\wedge...\wedge \hat{\delta}_{j}\wedge...\wedge \delta^{p})$; $2\leq p \leq n$. Where $\epsilon_{-1}(P\otimes\delta)=P\delta$.
\end{definition}

 \begin{proposition}\cite{FJCJN} \\
Let $f\in \mathcal{A}=\mathbb{C}[x,y]$ be a non zero reduced
polynomial. There exists a natural isomorphism
$\Omega_{\mathcal{A}}^{\bullet}(log
f)\overset{\simeq}{\longrightarrow}
\textbf{R}Hom_{A_{2}}(M^{log(f)}, \mathcal{A})$.
 \end{proposition}
 \begin{theorem}  \cite{FJCJN}\\
For any non zero reduced polynomial $f\in
\mathcal{A}=\mathbb{C}[x,y]$, the complexes
$(\Omega_{\mathcal{A}}^{\bullet}(log f), \epsilon_{\bullet}^{*})$
and $DR(\widetilde{M}^{log(f)}, \mathcal{A})$ are naturally
quasi-isomorphic.
 \end{theorem}
 For the inhomogeneous divisor $D= \{ h= 0\}$, with $h= xy+x^2y^2 + x^3y^3$, we deduce the following natural commutative diagram:

  \begin{eqnarray}
  \xymatrix{(\Omega_{\mathcal{A}}^{\bullet}(logD),\epsilon^{*}_{\bullet})  \ar[r]^. \ar[dr]_{}& \textbf{R} Hom_{\mathcal{A}}(M^{log(D)}, \mathcal{A})   \ar[d]^ r\\
                                    &             DR( \widetilde{M}^{log(D)},\epsilon^{ }_{\bullet})}
  \end{eqnarray}

 where $r : \mathcal{A} \longrightarrow \widetilde{M}^{log(D)} $ such that $a\mapsto r(a) = [ah]$ and $[.]$ represent the equivalent class of the ideal $A_{2}.\widetilde{Der}_{\mathcal{A}}(logD).$

\section{Kyoji Saito's basis of some free inhomogeneous divisors}
\begin{definition}
 A divisor $D=\{h=0\}$ is said to be inhomogeneous if there is not a vector field
 $\delta = \overset{p}{\underset{i =1}{\sum}}\varpi_{i}x_{i}\partial_{x_{i}}$ with $\varpi_{i}\in \mathbb{N}^{*}$ such that
  $\delta (h) =   \lambda  h$ or
     $h(t^{\varpi_{1}} x_{1},...,t^{\varpi_{p}} x_{p}) = t^{\varpi} h(x_{1},...,x_{p})$.
     On the other words $D$ is inhomogeneous divisor if it is not weighted homogeneous divisor.
  \end{definition}

We will now look for the Kyoji Saito's basis of the inhomogeneous
divisor given by $D= \{h= \overset{n}{\underset{i=1}{\sum}}
\alpha_{i}(\overset{p}{\underset{j=1,}{\Pi}} x_{j})^{n_{i}}  =0;
\alpha_{i}\in \mathbb{C}\}$ in dimension $p>3$. We show that, $D$ is
non quasi-homogeneous divisor. We remember that the homogeneous
divisors are the special cases of quasi-homogeneous divisors where
all the weights are equal to $1$.

 \begin{proposition}
  $D= \{ (x_{1},x_{2},...,x_{p})\in \mathbb{C}^{p}/~h= \overset{n}{\underset{i=1}{\sum}}
\alpha_{i}(\overset{p}{\underset{j=1,}{\Pi}} x_{j})^{n_{i}}  =0\}$ is an inhomogeneous divisor.
 \end{proposition}

\begin{proof}
We show that $D$ is not the quasi-homogenous divisor. But we suppose that it is a quasi-homogeneous divisor of degree $\varpi$ which respects  vector field
$\delta = \overset{n}{\underset{i=1}{\sum}}\varpi_{i}x_{i}\partial_{x_{i}}$ where each $\varpi_{i}\in \mathbb{N}^{*}$ is the weight of $x_{i}$.
On the other words we have $h(t^{\varpi_{1}} x_{1},...,t^{\varpi_{p}}x_{p})  = t^{\varpi} h( x_{1},...,x_{p})$
with $h=\alpha_{1} x_{1}^{n_{1}}x_{2}^{n_{1}}...x_{p}^{n_{1}} + \alpha_{2} x_{1}^{n_{2}}x_{2}^{n_{2}}...x_{p}^{n_{2}}+...+\alpha_{p} x_{1}^{n_{p}}x_{2}^{n_{p}}...x_{p}^{n_{p}}$, for all $t\in \mathbb{C}$ where $n_{i}$ are pairwise different. Polynomial $h$ is quasi-homogeneous if each monomial $M_{i}=\alpha_{i} x_{1}^{n_{i}}x_{2}^{n_{i}}...x_{p}^{n_{i}}$ has same degree $\varpi$. This imply that $n_{1}(\varpi_{1} + \varpi_{2}+...+\varpi_{p})=...=n_{p}(\varpi_{1} + \varpi_{2}+...+\varpi_{p})=\varpi$. This means that $n_{1}=...=n_{p}$.
 It is absurd because  $n_{i}$ are pairwise different. Hence, there is no positive integers  $\varpi_{i}$ (the weight of $x_{i}$)
 to make this above divisor  quasi-homogeneous or even weighted homogeneous. This completes the proof.
\end{proof}
We have locally  $hDer_{X,x} \subset Der_{X,x}(logD) \subset
Der_{X,x}$ according to \cite{KS}. Since
   $ Der_{X,x}(logD) $ and   $\Omega_{X,x}^{1}(logD)$ are reflexives $\mathcal{O}_{X,x}$-modules,
  we deduce that $\Omega_{X,x}^{1} \subset \Omega_{X,x}^{1}(logD) \subset \dfrac{1}{h}\Omega_{X,x}^{1}$.  This allow us to define in the following, the generator vector fields of $Der_{X,x}(logD)$ and consequently for $\Omega_{X,x}^{1}(logD)$.
  \begin{proposition}
 Let $E_{p}= \overset{p}{\underset{i=1}{\sum}} x_{i}\partial _{x_{i}}$ the local Euler vector field on $X$.
 The vector field given by $\delta_{n} =
\overset{n}{\underset{i=1}{\sum}}
\alpha_{i}(\overset{p}{\underset{j=1,}{\Pi}} x_{j})^{n_{i}-1} E_{p}$
is logarithmic along $D$.
  \end{proposition}
 \begin{proof}
Let $x$ an element of $D\subset X$ and $h_{x}$ the definition
hyper-plan of $D$.
 We get  $\delta_{n} (h_{x} )=
\overset{n}{\underset{i=1}{\sum}}
\alpha_{i}n_{i}(\overset{p}{\underset{j=1}{\Pi}}  x_{j})^{n_{i}-1}
\left(\overset{n}{\underset{i=1}{\sum}}
\alpha_{i}(\overset{p}{\underset{j=1}{\Pi}}  x_{j})^{n_{i}} \right)
= \overset{n}{\underset{i=1}{\sum}}
\alpha_{i}n_{i}(\overset{p}{\underset{j=1}{\Pi}}
x_{j})^{n_{i}-1}h_{x}$. It follows that $\delta_{n}
(h_{x})$ is in an ideal  $( h_{x})\mathcal{O}_{X,x}$. And then $\delta_{n}$ is logarithmic along
$D$.
 \end{proof}
  We consider $\overline{q}=(q_{1},...,q_{n}) \in \mathbb{N}^{n}$ with
  $\overset{n}{\underset{i=1}{\sum}}q_{i}=n$ and let us denote by  $\delta_{n}^{\overline{q}}=
 \overset{q_{1}}{\underset{i=1}{\sum}}
 \alpha_{i}(\overset{p}{\underset{j=1,}{\Pi}} x_{j})^{n_{i}-1}
 x_{1}\partial _{x_{1}}  +
 \overset{q_{2}}{\underset{i=q_{1}+1}{\sum}}
 \alpha_{i}(\overset{p}{\underset{j=1,}{\Pi}} x_{j})^{n_{i}-1}
 x_{2}\partial _{x_{2}}  +  ... +
 \overset{q_{n}}{\underset{i=q_{n-1}}{\sum}}
 \alpha_{i}(\overset{p}{\underset{j=1,}{\Pi}} x_{j})^{n_{i}-1}
 x_{n}\partial _{x_{n}}$ the vector field on $X$.

   \begin{proposition}
   The above vector field $\delta_{n}^{\overline{q}}$ is logarithmic along $D$.
   \end{proposition}

 \begin{proof}
 Let $x$ an element of $D$ and $h_{x}$ the definition hyper-plan of
 $D$.
  We get after computing that  $\delta_{n}^{\overline{q}} (h_{x}) = \delta (h_{x}) \in ( h_{x})\mathcal{O}_{X,x}$.
   According to the foregoing we can conclude that, the
 vector field $\delta_{n}^{\overline{q}}$ is logarithmic along
 $D$.
 \end{proof}

    $\mathcal{M} = \{ \delta^{0} \in Der_{X,x}(logD)~ / ~~\delta^{0}(h) = 0\}$ the submodule of $Der_{X,x}(logD)$. According to \cite{EF}, there is an isomorphism of $Der_{X,x}(logD)$ to $\mathcal{O}_{X,x}(\delta_{n} + \delta_{n}^{\overline{q}}) \oplus \mathcal{M}$. The  following lemma gives us  the explicit expression of $\mathcal{M}.$

 \begin{lemma}
  Let $(D,x)$ the above inhomogeneous divisor with defining the equation $h$.
  The submodule $\mathcal{M}$ of $Der_{X,x}(logD)$ is given by the following:
\begin{equation*}
  \mathcal{M} = \{  \overset{ }{\underset{0 < i ; j \leq n}{\sum}}\alpha_{ij}\delta^{0}_{ij} \in Der_{X,x}(logD), \alpha_{ij}\in \mathcal{A} /  \delta^{0}_{ij}  = (-1)^{i}x_{i}\partial x_{i} + (-1)^{i+1}x_{j}\partial x_{j} \}.
\end{equation*}

  \end{lemma}
  For example, let us take $h = \alpha_{1}x_{1}^{n_{1}}  x_{2}^{n_{1}}x_{3}^{n_{1}} +  \alpha_{2}x_{1}^{n_{2}}  x_{2}^{n_{2}}x_{3}^{n_{2}} + \alpha_{3}  x_{1}^{n_{3}}  x_{2}^{n_{3}}x_{3}^{n_{3}}$ on $\mathbb{C}[x_{1},x_{2},x_{3}]$. The vector fields $\delta^{0}_{12} = x_{1}\partial x_{1} - x_{2}\partial x_{2}$, $\delta^{0}_{13} = x_{1}\partial x_{1} - x_{3}\partial x_{3}$, $\delta^{0}_{23} = x_{2}\partial x_{2} - x_{3}\partial x_{3}$ and $\delta^{0} = \tilde{ \alpha}_{12} \delta^{0}_{12}  + \tilde{ \alpha}_{13} \delta^{0}_{13} + \tilde{ \alpha}_{23} \delta^{0}_{23}$
   for all $ \tilde{ \alpha}_{ij}= \alpha_{ij}- \alpha_{ji} \in \mathcal{A}$ are elements of $\mathcal{M}$. It's easy to verify that $\delta^{0}_{ij}(h)= \delta^{0}(h)=0$.
    A system of $p$ distinct vector fields $\delta_{n}^{\overline{q}}$ constitutes
    what we call the K. Saito basis of $Der_{\mathcal{A}}(logD)$, but it is not necessarily free basis. If this basis is free, we get that of $\Omega_{\mathcal{A}}^{1}(logD) $ by $p$ elements $\omega_{1}, ..., \omega_{p}$ where each  $\omega_{i}$ is the dual of $\delta^{i}$ which respect the following Saito criteria(\textbf{Theorem} $\ref{T1}$).

 \begin{corollary}  Let $\delta^{0} = \overset{}{\underset{0 < i ; j \leq p}{\sum}}\alpha_{ij}\delta^{0}_{ij}$, $\alpha_{ij}\in \mathcal{A}$:
\item[1.] The vector field $\delta^{0}$ is logarithmic along $D$ since all the vectors fields $\delta^{0}_{ij}$ are;
  \item[2.] The system $\mathcal{B} = \langle \delta^{1} ,...,  \delta^{p}\rangle$ is a free Saito's basis of  $Der_{\mathcal{A}}(logD)$ if and only if $\mathcal{B}^{*} = \langle  \omega_{1},...,\omega_{p} \rangle$ is also a free Saito's basis of $\Omega_{\mathcal{A}}^{1}(logD)$.
\end{corollary}

\subsection{Kyoji Saito's basis of $Der_{\mathcal{A}}(logD)$ for $h\in \mathcal{A} = \mathbb{C}[x_{1},x_{2},x_{3}]$}

\begin{proposition}
Let $D=\{h = \alpha_{1}x_{1}^{n_{1}}  x_{2}^{n_{1}}x_{3}^{n_{1}} +
\alpha_{2}x_{1}^{n_{2}}  x_{2}^{n_{2}}x_{3}^{n_{2}} + \alpha_{3}
x_{1}^{n_{3}}  x_{2}^{n_{3}}x_{3}^{n_{3}} = 0\}$.
 $\delta^{1}=x_{1}\partial_{x_{1}} - x_{2}\partial_{x_{2}}$, $\delta^{2}=x_{2}\partial_{x_{2}} - x_{3}\partial_{x_{3}}$ and $\delta^{3} = \delta_{3}^{1,1,1}$. Then the system $\mathcal{B} = \langle \delta^{1}, \delta^{2}, \delta^{3}\rangle$ constitutes what we call a free Saito basis of $Der_{\mathcal{A}}(logD)$.
\end{proposition}
              \begin{proof}
  We can firstly verify that $\mathcal{B} = \langle \delta^{1},\delta^{2},\delta^{3}  \rangle$
 is the Saito's basis of $Der_{\mathcal{A}}(logD)$ because, for all $i$ we have  $\delta^{i}(h)\in (h)$ with $\delta^{1}\wedge \delta^{2}\wedge \delta^{3} = h\partial _{x_{1}}\wedge \partial _{x_{2}} \wedge \partial _{x_{3}}$.
And secondly it is free because
$det(\textbf{A})=h$ is an unit multiple of $h$.  $\textbf{A}$ is called the Saito matrix of $D$ given by
     $ \textbf{A} = \left(
                    \begin{array}{ccc}
               x_{1}  & 0 &  \overline{a}_{1}\\
               -x_{2} & x_{2} &  \overline{a}_{2}\\
          0 & -x_{3} &  \overline{a}_{3}
                    \end{array}
                    \right)$ where $a_{i}= \alpha_{i}x_{1}^{n_{i}-1}  x_{2}^{n_{i}-1}x_{3}^{n_{i}-1}$, $\overline{a}_{i}=a_{i}x_{i}$.
                    The freedom of  $Der_{\mathcal{A}}(logD)$  comes from the fact that $det (\textbf{A}) = h$.
                     This completes the proof.
              \end{proof}

\subsection{Kyoji Saito's basis of $\Omega_{\mathcal{A}}^{1}(logD)$ with $h\in \mathcal{A} = \mathbb{C}[x_{1},x_{2},x_{3}]$}
The following proposition gives a basis of the dual of  the
$\mathcal{A}-$module $Der_{\mathcal{A}}(logD) $ noted by
$\Omega_{\mathcal{A}}^{1}(logD)$.
\begin{proposition}
We have $\Omega_{\mathcal{A}}^{1}(logD) = \langle \omega_{1}, \omega_{2}, \omega_{3}\rangle$ dual of $\langle \delta^{1}, \delta^{2}, \delta^{3}\rangle$ with:\\

$\omega_{1}  = \frac{x_{1}^{n_{2}-1}  x_{2}^{n_{2}}x_{3}^{n_{2}} + x_{1}^{n_{3}-1}  x_{2}^{n_{3}}x_{3}^{n_{3}}}{h}dx_{1} -  \frac{x_{1}^{n_{1}}  x_{2}^{n_{1}-1}x_{3}^{n_{1}}}{h}dx_{2}     -  \frac{x_{1}^{n_{1}}  x_{2}^{n_{1}}x_{3}^{n_{1}-1}}{h}dx_{3}$\\

$\omega_{2}  = \frac{x_{1}^{n_{3}-1}  x_{2}^{n_{3}}x_{3}^{n_{3}}}{h}dx_{1} -  \frac{x_{1}^{n_{3}}  x_{2}^{n_{3}-1}x_{3}^{n_{3}}}{h}dx_{2}     -  \frac{x_{1}^{n_{1}}  x_{2}^{n_{1}}x_{3}^{n_{1}-1} +  x_{1}^{n_{2}}  x_{2}^{n_{2}}x_{3}^{n_{2}-1}}{h}dx_{3}$\\

$  \omega_{3}  = \dfrac{x_{2}x_{3}}{h}dx_{1}  +
\dfrac{x_{1}x_{3}}{h}dx_{2}  +    \dfrac{x_{1}x_{2}}{h}dx_{3}.$
\end{proposition}

\begin{proof}
Let $\omega_{1}  = \omega_{1}^{1}dx_{1} + \omega_{1}^{2}dx_{2}    +
\omega_{1}^{3}dx_{3}$, $\omega_{2}  = \omega_{2}^{1}dx_{1} +
\omega_{2}^{2}dx_{2}    + \omega_{2}^{3}dx_{3} $  and $\omega_{3}  =
\omega_{3}^{1}dx_{1} + \omega_{3}^{2}dx_{2}    +
\omega_{3}^{3}dx_{3}$ elements of $\Omega_{\mathcal{A}}^{1}(logD)$.
According to  \textbf{proposition}  \ref{P3},
 the condition  $\delta^{1}\lrcorner \omega_{j }$ give us
  $\left(
  \begin{array}{ccc}
  \omega_{1}^{1}  & - \omega_{1}^{2} &  0\\
\omega_{2}^{1}  & - \omega_{2}^{2} &  0\\
\omega_{3}^{1}  & - \omega_{3}^{2} &  0\\
       \end{array}
       \right)   \left(
             \begin{array}{cc}
            x_{1}\\
            x_{2}\\
            x_{3}
             \end{array}
             \right) = \left(
                         \begin{array}{cc}
                        1\\
                        0\\
                        0
                         \end{array}
                         \right)$
                         with the same method,
    $\delta^{2}\lrcorner \omega_{j} $ implies
 $\left(
\begin{array}{ccc}
   0 &  \omega_{1}^{2} &  -\omega_{1}^{3} \\
  0  &  \omega_{2}^{2} &  - \omega_{2}^{3}\\
  0  &  \omega_{3}^{2} & - \omega_{3}^{3}\\
    \end{array}
      \right)   \left(
 \begin{array}{cc}
   x_{1}\\
 x_{2}\\
  x_{3}
  \end{array}
  \right) = \left(
  \begin{array}{cc}
  0\\
  1\\
 0
\end{array}
  \right)$
and
 $\delta^{3}\lrcorner \omega_{j} $ means that
 $\left(
\begin{array}{ccc}
 \tilde{a}_{11}  &\tilde{a}_{12}  &  \tilde{a}_{13} \\
 \tilde{a}_{21}    & \tilde{a}_{22} &   \tilde{a}_{23} \\
  \tilde{a}_{31}     & \tilde{a}_{32}  & \tilde{a} _{33} \\
    \end{array}
      \right)   \left(
 \begin{array}{cc}
   x_{1}\\
 x_{2}\\
  x_{3}
  \end{array}
  \right) = \left(
  \begin{array}{cc}
  0\\
  0\\
 1
\end{array}
  \right)$ where $\tilde{a}_{ij} = a_{j} \omega_{i}^{j} $.
Therefore we get any $\omega_{i}^{j}$ given in the previous
proposition. According to \textbf{theorem}  \ref{T1}  we   verify
the Saito criterion  $\omega_{1} \wedge\omega _{2}\wedge\omega _{3}=
\dfrac{1}{h} dx_{1}\wedge dx_{2}\wedge dx_{3}$. Since all
$h\omega_{j}$ are holomorphic, therefore we verify conditions given
in  \textbf{proposition} \ref{P2} that $h\omega$ and  $dh \wedge
\omega$ are holomorphics along $D$, as we have $d(h\omega) =
hd\omega + dh \wedge \omega.$ Furthermore $\delta^{i}\lrcorner
\omega_{j}= \delta_{ij}$. We conclude that $\mathcal{B}^{*} =
\langle \omega_{1}, \omega_{2}, \omega_{3}\rangle$ is a free Saito's
basis of $\Omega_{\mathcal{A}}^{1}(logD)$, dual of
$Der{\mathcal{A}}(logD)$.
\end{proof}

 \begin{proposition}  Let the Poisson algebra $( \mathcal{A} = \mathbb{C}[x_{1},x_{2}], \{-,-\}_h)$.
The Poisson bi-vector associated to the inhomogeneous Poisson
bracket   $\{x_{1},x_{2}\}_{h}=\overset{m}{\underset{i=1}{\sum}}
   \alpha_{i} x_{1}^{n_{i}} x_{2}^{n_{i}}$ is log-symplectic.
\end{proposition}
\begin{proof}
Consider the free inhomogeneous divisor $D= \{ h=0\} \subset X$ where $h=\overset{m}{\underset{i=1}{\sum}}
   \alpha_{i} x_{1}^{n_{i}} x_{2}^{n_{i}}$.
The bi-vector  $\pi=h\partial _{x_{1} }\wedge \partial _{x_{2}}$ is
the Poisson bi-vector  and its dual 2-form
$\omega=\dfrac{1}{h}dx_{1}\wedge dx_{2}$ is mesomorphic with simple
pole along $D$. Then $\pi$ is log-symplectic. This completes  the
proof.
\end{proof}

In the following, we focus on studying the case where the inhomogeneous
 divisor $h$ is an element of $\mathcal{A} = \mathbb{C}[x,y]$.

\section{On the  Koszul's bracket associate to $h\partial_x\wedge\partial_y$.}
\subsection{Saito's basis of $Der(\log \mathcal{I})$ and $\Omega_{\mathcal{A}}^{1}(log \mathcal{I}).$}
\begin{lemma} Let $h=xy+x^2y^2+x^3y^3\in \mathcal{A}=\mathbb{C}[x,y]$,
$\mathcal{I}=h\mathcal{A}$, $\label{31}$
\begin{equation}
  \mathcal{B}=( \delta^{1} = (x +
x^{2}y + x^{3}y^{2}) \partial_{x} + (y + xy^{2} +
x^{2}y^{3})\partial_{y},~ \delta^{2} = y \partial_{y} -
x\partial_{x})
\end{equation}
and
\begin{equation}
 \mathcal{B}^*=(\omega_{1} = \dfrac{y}{2h}dx +
\dfrac{x}{2h}dy,~ \omega_{2}= \dfrac{dy}{2y} - \dfrac{dx}{2x}).
\end{equation}
We have:
\begin{enumerate}
\item[1.] $\mathcal{B}$ is a Saito's basis of $Der_{\mathcal{A}}(log
\mathcal{I})$ and $\mathcal{B}^*$ is a Saito's basis of
$\Omega_{\mathcal{A}}^{1}(log \mathcal{I})$,
\item[2.] $\mathcal{B}$ and $\mathcal{B}^*$ are dual basis.
\end{enumerate}
\end{lemma}

\begin{proof}
We have the following
\begin{enumerate}
\item[1.] Since $\delta^{1}, \delta^{2} \in \mathcal{B}$, we have firstly verify that $\delta^{1} (h)= 2(1+2xy+3x^2y^2)h \in (h)$, $\delta^{2}(h) = 0 \in (h)$ and  secondly, $det(\delta^{1},\delta^{2}) = h$ is an unit multiple of $h$, then we conclude that $\mathcal{B}$  is the free Saito's basis of $Der_{\mathcal{A}}(log
\mathcal{I})$. Otherwise we have  $\omega_{1},~ \omega_{2} \in
\mathcal{B}^*$ with $\omega_{1} \wedge \omega_{2}=
\dfrac{(\frac{1}{2})}{h} dx \wedge dy =  \dfrac{unit}{h} dx \wedge
dy $. Its well known that $d(h\omega) = hd\omega + dh \wedge
\omega.$ Furthermore, since $h\omega_{j}$ are holomorphic for all
$j$, therefore $h\omega$ and  $dh \wedge \omega$ are holomorphic
along $\mathcal{I}$.
 Then the set of 1-forms defined by $\mathcal{B}^*$  is a Saito's basis of
$\Omega_{\mathcal{A}}^{1}(log \mathcal{I})$.
\item[2.]  $\mathcal{B}$ and $\mathcal{B}^*$ are the dual basis because we
 can easily verify that $\delta^{i} \lrcorner \omega_{j}=\delta_{ij}$ for
  all $i,j=1;2$ where $\delta_{ij}$ is the Kronecker's symbol.
\end{enumerate}
This completes the proof.
\end{proof}
\subsection{Lie-Rinehart structure on $\Omega_{\mathcal{A}}^{1}(log \mathcal{I})$}

\begin{proposition}
 The Poisson structure
$\{-,-\}_h$ induce  an homomorphism of $\mathcal{A}-$modules given by
$\tilde{H}: \Omega_{\mathcal{A}}^{1}(log\mathcal{I}) \longrightarrow
Der_{\mathcal{A}}(log \mathcal{I})$  such that, for all $\omega =
a\omega_{1} + b \omega_{2}\in \Omega_{\mathcal{A}}^{1}(log\mathcal{I})$, $a,b \in \mathcal{A}=\mathbb{C}[x,y]$,
 we have the following expression:
 \begin{equation}
             \tilde{H}(\omega)=   \frac{1}{2}(a \delta^{2}- b \delta^{1}) \in Der_{\mathcal{A} }(Log \mathcal{I}).
      \end{equation}
\end{proposition}
\begin{proof}
According to \cite{IV}, for any function $f\in C^{\infty}$;
$\{f,-\}$ is a derivation of $C^{\infty}$ and there exists a well
defined vector field $\mathcal{X}_{f}:=H_{f}\in
Hom(\Omega_{\mathcal{A}}^{1},  Der_{\mathcal{A}})$ such that: for
all $g\in C^{\infty}$ we have  $\{f,g\} = H_{f} (g)=dg(H_{f})$. This
vector field $H_{f}$ is called the Hamiltonian vector fields(or the
hamiltonian map).   Let $H$ be the hamiltonian map induced by the
        Poisson structure $\{x,y\} = h$. Since $Im (H) \subset Der_{\mathcal{A}} = \langle \partial_{x}, \partial_{y} \rangle$, then for all $x,y \in \mathcal{A}$ we get the following system:
           \begin{eqnarray*}
          \left\{
           \begin{array}{rl}
           H(dx)= x_{1}\partial_{x}+x_{2}\partial_{y}\\
           H(dy)= y_{1}\partial_{x}+y_{2}\partial_{y}
           \end{array}
           \right.
           .~~~~~~~~~~~~On ~ the ~other ~words~we ~have
           \end{eqnarray*}
         \begin{eqnarray*}
       \left\{
          \begin{array}{rl}
          \{x,x\}_{h} = H(dx)(x)= x_{1}\\
          \{x,y\}_{h}=H(dx)(y)= x_{2}
          \end{array}
          \right.
          and ~~  \left\{
           \begin{array}{rl}
           \{y,x\}_{h} =H(dy)(x)= y_{1}\\
            \{y,y\}_{h}=H(dy)(y)= y_{2}
           \end{array}
           \right.
          \end{eqnarray*}
       \begin{eqnarray*}
        Then~~ we~~have: \left\{
           \begin{array}{rl}
           0= x_{1};\\
           \{x,y\}_{h} = x_{2}
           \end{array}
           \right.
           and ~~  \left\{
            \begin{array}{rl}
            -\{x,y\}_{h} = y_{1}\\
             0= y_{2}
            \end{array}
            \right.
           \end{eqnarray*}
         Hence, we get $H(dx)=\{x,y\}_{h}\partial_{y} $, $H(dy)= -\{x,y\}_{h}\partial_{x}$ and it is clear that for all $f\in \mathcal{A}$, $H_{f}=h(\dfrac{\partial f}{\partial_{x}}\dfrac{\partial }{\partial_{y}} -  \dfrac{\partial f}{\partial_{y}}\dfrac{\partial }{\partial_{x}})$.
          We remember that   $\delta^{1} = (x + x^{2}y+ x^{3}x^{2})\partial_{x} + (y + xy^{2}+ x^{2}x^{3})\partial_{y}$, $\delta^{2} = y\partial_{y} - x\partial _{x}$, $\omega_{1} = \dfrac{y}{2h} dx+ \dfrac{x}{2h}dy$ and  $\omega_{2} = \dfrac{dy}{2y} - \dfrac{ dx}{2x}$.
It's well known that the $\mathcal{A}-$modules $\Omega_{\mathcal{A}}^{1}(log\mathcal{I}) $ and $Der_{\mathcal{A}}(log \mathcal{I})$ are respectively the submodules of $\Omega_{\mathcal{A}} $ and $ Der_{\mathcal{A}}$. Moreover, for all  $\omega \in \Omega_{\mathcal{A}}^{1}(log\mathcal{I})$ we have $H(\omega) \in Der_{\mathcal{A}}(log \mathcal{I})$. On the other words $H(\Omega_{\mathcal{A}}^{1}(log\mathcal{I}) ) \subset Der_{\mathcal{A}}(log \mathcal{I})$.
         Then this structure induce also  another homomorphism of $\mathcal{A}-$module
        $\tilde{H}\in Hom( \Omega_{\mathcal{A}}^{1}(log\mathcal{I}), Der_{\mathcal{A}}(log \mathcal{I}))$
        that we call the logarithmic hamiltonian map.
       We will determine the explicit expression of $\tilde{H}: \Omega_{\mathcal{A}}^{1}(log\mathcal{I}) \longrightarrow   Der_{\mathcal{A}}(log \mathcal{I})$.
Since $\tilde{H}(\omega_{1})  = \tilde{H}(\dfrac{y}{2h}dx + \dfrac{x}{2h}dy)$, by the $\mathcal{A}$-linearity of $\tilde{H}$,
           We have  $\tilde{H}(\omega_{1})   =  \dfrac{y}{2h}H(dx) + \dfrac{x}{2h}H(dy)$. It's clearly follows that
  $\tilde{H}(\omega_{1}) =  \dfrac{y}{2h}(h\partial_{y}) - \dfrac{x}{2h}(h\partial_{x})$,
   from which we get
   $\tilde{H}(\omega_{1}) =  \dfrac{1}{2}\delta^{2}.$
We use the same method to show that $\tilde{H}(\omega_{2}) =   - \dfrac{1}{2}\delta^{1}$.
Since $\tilde{H}$ is $\mathcal{A}$-linear, we
 deduce that for all $\omega= a \omega_{1} + b\omega_{2} \in \Omega_{\mathcal{A} }^{1}(log \mathcal{I})$ with $a, b \in \mathcal{A}$, we have $\tilde{H}(\omega)= a \tilde{H}(\omega_{1}) + b\tilde{H}(\omega_{2})$. Then, the expression of the  hamiltonian logarithmic vector field $\tilde{H}$ is given as follows:
      \begin{equation}
             \tilde{H}(\omega)=   \frac{1}{2}(a \delta^{2}- b \delta^{1}) \in Der_{\mathcal{A} }(log \mathcal{I}) ~for ~all~a,~b\in \mathcal{A}.
      \end{equation}
      This complete the proof of the above proposition.
\end{proof}
To complete the construction of the logarithmic Lie-Rinehart-Poisson structure on
$\Omega_{\mathcal{A}}^{1}(log\mathcal{I})$, it is necessary to have the Koszul's bracket expression.
\begin{lemma} \label{(30)}
  Let $[-,-]_{\Omega_{\mathcal{A}}^{1}(log\mathcal{I})} $  the Koszul's bracket  induced on $\Omega_{\mathcal{A}}^{1}(log\mathcal{I})$ by
  the operator $\pi = h\partial x\wedge \partial y$. We have $[\omega_{1},\omega_{2}]_{\Omega_{\mathcal{A}}^{1}(log\mathcal{I})}=0$.
\end{lemma}

 \begin{proof}
According to \cite{DJ1}, the Koszul's bracket of two 1-forms on
$\Omega_{\mathcal{A}}^{1}(log\mathcal{I})$  is:
     \begin{equation} \label{4}
       [\omega_{1},\omega_{2}]_{\Omega_{\mathcal{A}}^{1}(log\mathcal{I})} = \mathcal{L}_{\tilde{H}(\omega_{1})}(\omega_{2}) -\mathcal{L}_{\tilde{H}(\omega_{2})}(\omega_{1})-d\pi(\omega_{1},\omega_{2}).
     \end{equation}
  Consider the inner product $\mathit{i}_{v}:  \mathcal{L}alt( \Omega_{\mathcal{A}}^{1}(log\mathcal{I}), \mathcal{A}) \longrightarrow  \mathcal{L}alt( \Omega_{\mathcal{A}}^{1}(log\mathcal{I}), \mathcal{A})$ such that $(\mathit{i}_{v}f)(x, y) = f(v, x, y)$.
  According to \cite{TM}, Lie's field of derivation  $\mathcal{L}_{v}$ on the field $v$ verify:
 \begin{equation} \label{5}
      \mathcal{L}_{v}= \mathit{i}_{v} \circ d + d\circ \mathit{i}_{v}.
      \end{equation}
     For all 1-forms $\eta_{1}= \eta_{1}^{1} \omega_{1} + \eta_{1}^{2} \omega_{2},\eta_{2} = \eta_{2}^{1} \omega_{1} + \eta_{2}^{2} \omega_{2}$ we can also define the above Koszul's bracket   as follows:
       $[-,-]_{\Omega_{\mathcal{A}}^{1}(log\mathcal{I})}: \Omega_{\mathcal{A}}^{1}(log \mathcal{I}) \times  \Omega_{\mathcal{A}}^{1}(log \mathcal{I}) \longrightarrow  \Omega_{\mathcal{A}}^{1}(log \mathcal{I}) $ such that $[\eta_{1},\eta_{2}]_{\Omega_{\mathcal{A}}^{1}(log\mathcal{I})} = \psi^{1}  \omega _{1} + \psi^{2}  \omega _{2} \in  \Omega_{\mathcal{A}}^{1}(log \mathcal{I})$ where $\psi^{1}$ and $ \psi^{2}$ are such that $\langle [\eta_{1}, \eta_{2}]_{\Omega_{\mathcal{A}}^{1}(log \mathcal{I})}, \omega_{i}^{*} \rangle = \psi^{i} \in Der_{\mathcal{A}}(log \mathcal{I})$ with $\omega_{i}^{*}=\delta^{i}$.
      Furthermore  we have $[\omega_{1},\omega_{2}]_{\Omega_{\mathcal{A}}^{1}(log\mathcal{I})} = \phi^{1}  \omega _{1} + \phi^{2}  \omega _{2}$; $\phi^{i} \in Der_{\mathcal{A}}(log \mathcal{I})$.
       It follows from equations ($\ref{4}$) and ($\ref{5}$) that:
       \begin{equation}
       \left(  \mathcal{L}_{\tilde{H}(\omega_{1})}(\omega_{2}) -\mathcal{L}_{\tilde{H}(\omega_{2})}(\omega_{1})-d\pi(\omega_{1},\omega_{2}) \right)(\omega_{1}^{*}) = \phi^{1}
       \end{equation}

       \begin{equation}
        \left(  \mathcal{L}_{\tilde{H}(\omega_{1})}(\omega_{2}) -\mathcal{L}_{\tilde{H}(\omega_{2})}(\omega_{1})-d\pi(\omega_{1},\omega_{2}) \right)(\omega_{2}^{*}) = \phi^{2}
       \end{equation}
   After computing, we get with (\ref{5}) that $\mathcal{L}_{\tilde{H}(\omega_{1})}(\omega_{2}) = \mathcal{L}_{\tilde{H}(\omega_{2})}(\omega_{1})=0$. On the other hand we have $\pi (\omega_{1},\omega_{2}) = h(\dfrac{y}{2h}.\dfrac{1}{2y} - \dfrac{x}{2h}.(-\dfrac{1}{2x}))= \dfrac{1}{2}$ implies that
  $d\pi (\omega_{1}, \omega_{2})=0$. Therefore $\phi^{i} =0$ and it follows that $[\omega_{1},\omega_{2}]_{\Omega_{\mathcal{A}}^{1}(log\mathcal{I})} =0$. This completes the proof.
 \end{proof}

\section{Application: Logarithmic Cohomology of $\{-,-\}_h$ where  $h\in \mathcal{A}=\mathbb{C}[x,y]$}

\subsection{ Logarithmic Poisson Cohomology of $\{-,-\}_h.$}
According to \textit{remark} \ref{R2}, we define the logarithmic Poisson cohomology as follows:
 \begin{definition}
             $H_{log}^{i}(\mathcal{A};\{-,-\}_{h}) = \dfrac{Kerd^{i+1}_{\tilde{H}}}{Imd^{i}_{\tilde{H}}}$ for all $i\geqslant 0$ are called
             $i^{th}$ group of logarithmic Poisson cohomology of $\{-,-\}_h$.

          $Kerd^{i}_{\tilde{H}}$ is the space of Poisson cocycles of order $i$ and  $Imd^{i}_{\tilde{H}}$ is the space of
                   Poisson coboundary of order  $i$. We will simply denote $H_{log}^{i}(\mathcal{A};\{-,-\}_{h}) $ by $H_{log}^{i} $.
            \end{definition}
The logarithmic  cochain complex of Poisson associate to our particular type divisor in dimension two is given by the following:
     \begin{displaymath}
     0\stackrel{d^{0}_{\tilde{H}}}{\longrightarrow}\mathcal{A}\stackrel{d^{1}_{\tilde{H}}}{\longrightarrow}Der_{\mathcal{A}}(log \mathcal{I})
     \stackrel{d^{2}_{\tilde{H}}}{\longrightarrow} \wedge^{2}Der_{\mathcal{A}}(log \mathcal{I}) \stackrel{d^{3}_{\tilde{H}}}{\longrightarrow} 0.
     \end{displaymath}
      The logarithmic Poisson differentials are  $d^{0}_{\tilde{H}} = d^{3}_{\tilde{H}} = 0$ where $d^{1}_{\tilde{H}}$ and $d^{2}_{\tilde{H}}$:
 \begin{lemma} \label{41}
 The logarithmic  Poisson differentials  $d^{1}_{\tilde{H}}$ and $d^{2}_{\tilde{H}}$ are given as follows:
 \begin{enumerate}
 \item[1.] $d^{1}_{\tilde{H}}(f)=( \delta^{2}f) \delta^{1} +(- \delta^{1}f) \delta^{2} \in Der_{\mathcal{A}}(log \mathcal{I})$ for all $f\in \mathcal{A}$,
 \item[2.] $ d^{2}_{\tilde{H}}(\overset{\longrightarrow}{f})=( \delta^{1}f^{1}+ \delta^{2}f^{2})\delta^{1} \wedge \delta^{2}$ for all  $\overset{\longrightarrow}{f} = f^{1}\delta^{1} + f^{2}\delta^{2}\in Der_{\mathcal{A}}(log \mathcal{I})$.
 \end{enumerate}
 \end{lemma}

 \begin{proof}
 Let  $f\in \mathcal{A}$. On the one hand,  $d^{1}_{\tilde{H}}(f)=f^{1}\delta^{1}+f^{2}\delta^{2}$.  It's the reason to say that $f^{1}=d^{1}_{\tilde{H}}(f)\lrcorner \omega_{1}$ and $f^{2}=d^{1}_{\tilde{H}}(f)\lrcorner \omega_{2}$. On the other words we have
     $ f^{1} =  \dfrac{1}{2}\delta^{2}f$
 and  $ f^{2} =  -\dfrac{1}{2}\delta^{1}f.$   Therefore $d^{1}_{\tilde{H}}(f)=\dfrac{1}{2}(\delta^{2}f\delta^{1}  - \delta^{1}f\delta^{2})$. It is equivalent to
  $d^{1}_{\tilde{H}}(f)=\delta^{2}f\delta^{1} + (- \delta^{1}f)\delta^{2}$. On the other hand, we have for all $\overset{\longrightarrow}{f}=f^{1}\delta^{1}+f^{2}\delta^{2}\in Der_{\mathcal{A}}(log \mathcal{I})$,  $ d^{2}_{\tilde{H}}(\overset{\longrightarrow}{f})=  d^{2}_{\tilde{H}}(f^1\delta^{1}  + f^2\delta^{2})$. Since
 $d^{2}_{\tilde{H}}$ is $\mathcal{A}$-linear, hence
  $ d^{2}_{\tilde{H}}(\overset{\longrightarrow}{f})=  d^{2}_{\tilde{H}}(f^1\delta^{1})  + d^{2}_{\tilde{H}}(f^2\delta^{2}) \in \wedge^{2}Der_{\mathcal{A}}(log\mathcal{I})$.  So we can say that $ d^{2}_{\tilde{H}}(f^1\delta^{1})$ and  $ d^{2}_{\tilde{H}}(f^2\delta^{2})$ live in $\wedge^{2}Der_{\mathcal{A}}(log\mathcal{I}) $.
 Then, these  logarithmic Poisson differentials can therefore be written in the following forms: $  d^{2}_{\tilde{H}}(f^1\delta^{1})= a \delta^{1} \wedge \delta^{2}_{\tilde{H}}$ and $ d^{2}_{\tilde{H}}(f^2\delta^{2}) = b \delta^{1} \wedge \delta^{2}$ where $a,b \in \mathcal{A}$ with
 $a = d^{2}_{\tilde{H}}(f^1\delta^{1}) \lrcorner (\omega_{1}, \omega_{2})$
   and $b = d^{2}_{\tilde{H}}(f^2\delta^{2}) \lrcorner (\omega_{1}, \omega_{2})$ are (according to the above relation (R)) such that
    \begin{align*}
     a &\ = d_{\tilde{H}}^{2}(f^1\delta^{1}) \lrcorner (\omega_{1}, \omega_{2})\\
     &\ = -\tilde{H}(\omega_{1})f^1\delta^{1}\lrcorner \omega_{2} + \tilde{H}(\omega_{2})f^1\delta^{1}\lrcorner \omega_{1} - f^1\delta^{1}\lrcorner [\omega_{1},\omega_{2}]_{\Omega_{{\mathcal{A}}(log \mathcal{I})}^{1}}\\
      &\ =  \tilde{H}(\omega_{2})f^1,~~since ~~\delta^{i}\lrcorner \omega_{j} = \delta_{ij}~and~[\omega_{1},\omega_{2}]_{\Omega_{{\mathcal{A}}(log \mathcal{I})}^{1}}=0.
     \end{align*}
      With the same method,
     $b = d^{2}_{\tilde{H}}(f^2\delta^{2}) \lrcorner (\omega_{1}, \omega_{2}) = - \dfrac{1}{2}\delta^{2}f^2$.
    So we find the expression of $ d^{2}_{\tilde{H}}$ given by  $ d^{2}_{\tilde{H}}(\overset{\longrightarrow}{f})=  -\dfrac{1}{2}(\delta^{1}f^1 +  \delta^{2}f^2)\delta^{1}\wedge  \delta^{2}$ for all $\overset{\longrightarrow}{f} \in Der_{\mathcal{A}}(log \mathcal{I})$.
    We will consider an equivalent expression that is $ d^{2}_{\tilde{H}}(\overset{\longrightarrow}{f})=( \delta^{1}f^1 +  \delta^{2}f^2) \delta^{1}\wedge  \delta^{2}$.
    This completes the proof.
 \end{proof}

 Let $\partial^{0}$, $\partial^{1}$, $\partial^{2}$ be the canonical isomorphisms  respectively given by the following:
    $\partial^{0}:= Id_{\mathcal{A}}:    \mathcal{A} \longrightarrow  \mathcal{A}$, $a \longmapsto \partial^{0} (a) = a$(an identity map);\\
     $ \partial^{1}:   \mathcal{A}^{2}    \longrightarrow Der_{\mathcal{A}}(log \mathcal{I}), ~~
                          ( f^1 , f^2) \longmapsto \partial^{1}(f^1, f^2) = \delta^{1} f^1  + \delta^{2} f^2$;\\
     $\partial^{2}:  \mathcal{A}    \longrightarrow  \wedge^{2}Der_{\mathcal{A}}(log \mathcal{I}),~~
                                                                                   f \longmapsto \partial^{2}(f) = f\delta^{1}  \wedge \delta^{2}.$
    So its follows that  $\partial^{1}\circ d^{1} = d^{1}_{\tilde{H}} \circ \partial^{0}$, $\partial^{2}\circ d^{2} = d^{2}_{\tilde{H}} \circ \partial^{1}$ and this makes the diagram below commutative:

  \[ \xymatrix{ 0\ar[r]^{d^0}& \mathcal{A} \ar[d]^{\partial^{0}}\ar[r] ^{d^1} &\mathcal{A}\times \mathcal{A}\ar[d]^{\partial^1}\ar[r]^{d^2}&\mathcal{A} \ar[d]^{\partial^2}\ar[r] ^{d^3}&0\\
      0\ar[r] & \mathcal{A} \ar[r]&Der_{\mathcal{A}}(log \mathcal{I})\ar[r]&\bigwedge^{2} Der_{\mathcal{A}}(log \mathcal{I})\ar[r]&0}
     \]

 The following lemma informs us about the corresponding cochain complex.
 \begin{lemma} $\label{43}$
  $ \xymatrix{ 0 \ar[r]^{d^{0}} & \mathcal{A} \ar[r]^{d^{1}}  & \mathcal{A}\times \mathcal{A}
            \ar[r]^{d^{2}}  & \mathcal{A} \ar[r]^{d^{3}}  & 0 }
            $
           is the cochain complex.
 \end{lemma}
 \begin{proof}
     Indeed, for all $f\in \mathcal{A}$ and for all $\overset{\longrightarrow}{f}= (f^1,f^2) \in Der_{\mathcal{A}}(log \mathcal{I})$, we have by definition  $d^{1}(f)= (\delta^{2}f, - \delta ^{1}f)$ and $d^{2}(\overset{\longrightarrow}{f})=\delta^{1}f^1+ \delta^{2}f^2$.  Firstly, according to \textbf{proposition} $\ref{P3}$, $ [\delta^{1},\delta^{2}]_{Der_{\mathcal{A}}(log \mathcal{I})}=0$. Then $d^{2}\circ d^{1} = 0$. This completes the proof of the above lemma.
     \end{proof}

  \begin{lemma} \label{45} Let $C^{i}, i\geqslant 0$ the $\mathcal{A}$-modules and  $H^{i}$ the submodules of $C^{i+1}$ with the cochain complex given by $........ C^{i}\stackrel{d^{i}} \longrightarrow C^{i+1}\stackrel{d^{i+1}}{\longrightarrow} C^{i+2} ........ $
  If for all $i \geqslant 0$, we have $ Ker~d^{i+1} = Im~d^{i} \oplus H^{i}$; then $H^{i} \simeq H_{log}^{i}$.
  \end{lemma}

   We  use the above lemma to explicitly compute  the logarithmic Poisson cohomology of $\{-,-\}_{h}$. The results are recorded in the following theorem.
      \begin{theorem} The logarithmic Poisson  cohomology groups of
      $\{-,-\}_h$
      along $\mathcal{I}$ are given by:
      \begin{enumerate}
      \item[1.]
      $H_{log}^{0}   \simeq   \mathbb{C}$,
       \item[2.]
       $H_{log}^{1}   \simeq \mathbb{C}\times (\mathbb{C} \oplus xy \mathbb{C})$,
       \item[3.]  $H_{log}^{2}  \simeq   \mathbb{C} \oplus xy \mathbb{C}$,
      \end{enumerate}
      \end{theorem}
       The  proof of this theorem  requires that of the following lemmas:

 \begin{lemma} The 2-cocycle associate to the inhomogeneous Poisson structure $\{-,-\}_{h}$ is given by
      $Ker d^{2} \simeq \mathbb{C} \times \mathbb{C}_{n}[x,y]$, where  $\mathbb{C}_{n}[x,y] = \overset{}{\underset{n\geq 0}{\bigoplus}} c_{n}x^{n}y^{n}; c_{n} \in \mathbb{C}$.
     \end{lemma}

  \begin{proof}
  Let $\overset{\longrightarrow}{f} = (f^1,f^2) \in Ker d^{2}$. We  consider two cases, to explicitly identify  the module that contains $f^1$ and $f^2$. Firstly, if $f^2=f^2 _{ij}x^{i}y^{j}$ with $i\neq j$; according to the (\textit{lemma} \ref{41}) and  the (\textit{lemma} \ref{43}), we have $(f^1,f^2) \in Ker d^{2}$ is equivalent to
   $\left[ (i+j)f^{1}_{ij} +(j-i)f^{2}_{ij}  \right] + (i+j)f^{1}_{ij}(xy+x^2y^2) =0$, consequently we have
   \begin{eqnarray*}
        \left\{
          \begin{array}{rl}
     f^{2}_{ij}= \dfrac{i+j}{i-j} f^{1}_{ij};~~ i\neq j    ~~~~~~\hfill  (s_{1})\\
       (i+j)f^{1}_{ij}=0 ~~~~~~~~ \hfill(s_{2})
        \end{array}
          \right.
         .~~ In~ (s_{2}) ~we ~have~f^{1}_{ij}=0.~Hence ~f^{2}_{ij}=0.
       \end{eqnarray*}
       And secondly, if  $f^2=f^2 _{ii}x^{i}y^{i}$, we have $d^{2}(\overset{\longrightarrow}{f}) = 0$ means that $\delta^{1}(f^1)=0$. On the other words $x\partial _{x}f^{1} +  y\partial _{y}f^{1}=0$. The solutions to this partial derivative equation are constants. Therefore we get $\overset{\longrightarrow}{f}\in \mathbb{C} \times \mathbb{C}_{n}[x,y]$.
       It follows that $Ker d^{2} \subset \mathbb{C} \times \mathbb{C}_{n}[x,y]$. Conversely, we  verify that $  \mathbb{C} \times \mathbb{C}_{n}[x,y]   \subset Ker d^{2}$, because $d^{2}(c, f_{i}x^{i}y^{i}) = \delta^{2}(f_{i}x^{i}y^{i})=0$ for all $c;f_{i}\in \mathbb{C}$.
 Finally we have $Ker d^{2} \simeq \mathbb{C} \times \mathbb{C}_{n}[x,y]$. This completes the proof of the above lemma.
   \end{proof}
 In the following, we note the splayed free inhomogeneous divisor $h=0$ as $h=xyh_{xy}=0$, where $h_{xy}= 1 + xy+ x^{2} y^{2}$.
 \begin{lemma} The 1-cobord of inhomogeneous  Poisson bracket $\{-,-\}_{h}$ is given by
     $Im d^{1}  \simeq 0\times h_{xy}\mathbb{C}_{n}[x,y]$.
     \end{lemma}
     \begin{proof} Let $\overset{\longrightarrow}{\varphi} = \varphi^{1}\delta^{1}+\varphi^{2}\delta^{2}
     \in Der_{\mathcal{A}}(log \mathcal{I})$, with
      $\overset{\longrightarrow}{\varphi}\in Imd^{1}$.
   This implies that there exists  $f\in \mathcal{A}$ such that $d^{1}(f)=
       \overset{\longrightarrow}{\varphi}$. On the other words,
      according to the  (\textit{lemma} \ref{41} ) and (\textit{lemma} \ref{43}) we have $d^{1}(f) \in
      \mathcal{A} \times h_{xy}\mathcal{A}$.
       Let us take $\varphi^{1}= a \in
       \mathcal{A}$ and $ \varphi^{2}= h_{xy}b \in h_{xy}\mathcal{A}$.
       It is follows that:
        \begin{eqnarray*}
             \left\{
                \begin{array}{rl}
                \delta^{2} f = a\\
                      -\delta^{1} f = h_{xy}b.
                \end{array}
                \right.
               Then~we~have   ~~  \left\{
                 \begin{array}{rl}
                y\partial_{y} f -x\partial_{x}  f = a \\
                             -(x \partial_{x} f + y\partial _{y}) f = b
                 \end{array}
                 \right.
                \end{eqnarray*}
 The above systems are equivalent to the following systems:
        \begin{eqnarray*}
                \left\{
                   \begin{array}{rl}
                  (j-i) f_{ij} = a_{ij} \\
                               -(i+j)f_{ij} = b_{ij}
                   \end{array}
                   \right.
                    if~and~only~if  ~~  \left\{
                               \begin{array}{rl}
                              (j-i) f_{ij} x^{i}y^{j} = \varphi^{1}\\
                -(i+j)h_{xy}f_{ij}x^{i}y^{j} =  \varphi^{2}
                               \end{array}
                               \right.
                  \end{eqnarray*}

                  From there, $d^{1}(f)= \overset{\longrightarrow}{\varphi} \in  \underset{i,j\geq 0}{\bigoplus} \left( j-i; ~~-(i+j)h_{xy} \right) x^{i}y^{j}\mathbb{C}$.

    We just showed that $ Im d^{1} \subset \underset{i,j\geq 0}{\bigoplus} \left(j-i; ~~-(i+j)h_{xy}\right)
     x^{i}y^{j}\mathbb{C}$.
     Since $Im d^{1}\subset Ker d^{2}   \simeq \mathbb{C} \times \mathbb{C}_{n}[x,y]$, we deduce that $i=j$.
     Consequently $ Im d^{1} \subset\underset{i\geq 1}{\bigoplus} \left( 0;-2ih_{xy} \right) x^{i}y^{i}\mathbb{C}$.
     And conversely we will take $p\neq (0;0)$
     such that $ p \in  \underset{i\geq 1}{\bigoplus} \left( 0;-2ih_{xy} \right) x^{i}y^{i}\mathbb{C}$
      and we shall prove that $p \in Im d^{1}$. Let $p = (0; -2ih_{xy}p_{ii}x^{i}y^{i})$.
      We are looking for $0\neq g\in \mathcal{A}$ so that we are $d^{1}(g) = p$.
      On the other words we have the following:
     \begin{eqnarray*}
                \left\{
                   \begin{array}{rl}
                  (j-i) g_{ij} = 0 \\
                               -(i+j)  g_{ij} = -2i p_{ii}
                   \end{array}
                   \right.
                  \end{eqnarray*}
            Then $i=j$ for all $g_{ij}\neq 0$ and the existence of $g$ just comes by taking $g_{ii}=p_{ii}$.
            We have showed that there exits $g = g_{ii}x^{i}y^{i}\in \mathcal{A}$ such that $d^{1}(g) = p$.
            Therefore $p \in Im d^{1}$. Hence $  \underset{i\geq 1}{\bigoplus} \left(0; -2ih_{xy}
             \right) x^{i}y^{i}\mathbb{C}  \subset  Im d^{1}$.
              These both inclusions allow us to conclude that $Im d^{1} =  \underset{i\geq 0}{\bigoplus} \left(0; -2ih_{xy} \right) x^{i}y^{i}\mathbb{C}$.
     And as we have $\underset{i\geq 0}{\bigoplus} \left(0; -2ih_{xy} \right) x^{i}y^{i}\mathbb{C}\simeq \underset{i\geq 0}{\bigoplus} \left(0; h_{xy} \right) x^{i}y^{i}\mathbb{C}$, then we can deduce that
     $Im d^{1}\simeq \underset{i\geq 0}{\bigoplus} \left(0; h_{xy}\right) x^{i}y^{i}\mathbb{C} \simeq 0 \times h_{xy}\mathbb{C}_{n}[x,y]$.
     This completes the proof of the above lemma.
     \end{proof}
 We will now determine the explicit expression of 2-cobord noted by $Im d^{2}$.
    \begin{lemma} \label{L9}
   The 2-cobord associate to the inhomogeneous  Poisson bracket $\{-,-\}_{h}$ is given by
                  $Im d^{2} \simeq \underset{i,j\geq 0}{\bigoplus}h_{xy} x^{i}y^{j}\mathbb{C} +
                             \underset{\overset{i,j\geq 0}{i\neq j}}{\bigoplus} x^{i}y^{j}\mathbb{C}$.
    \end{lemma}
    \begin{proof}
        Let $\varphi \in Im d^{2}$. So there  exists $f^{1}\delta^{1} + f^{2}\delta^{2}= \overset{\longrightarrow}{f}\in Der_{\mathcal{A}}(log \mathcal{I})$ such that $d^{2}(f^{1}, f^{2}) = \varphi$. We also have  $d^{2}(\mathcal{A}\times \mathcal{A}) \subseteq h_{xy} \mathcal{A} + \mathcal{A}$. Take for this reason $\varphi = h_{xy} \varphi^{1} + \varphi^{2}$, $\varphi^{i}\in \mathcal{A}$.  $d^{2}(f^{1}, f^{2}) = \varphi$ it is equivalent to $\delta^{1}f^{1} + \delta^{2}f^{2} = h_{xy} \varphi^{1} + \varphi^{2}$. It means on the other words that   $h_{xy}(x\partial_{x}f^{1} + y\partial_{y}f^{1}) + (y\partial_{y}f^{2} - x\partial_{x}f^{2})  = h_{xy} \varphi^{1} + \varphi^{2}$.
        In this case we have the following:
     \begin{equation} \label{36}
     h_{xy}(x\partial_{x}f^{1} + y\partial_{y}f^{1} - \varphi^{1}) = x\partial_{x}f^{2} - y\partial_{y}f^{2} + \varphi^{2}.
     \end{equation}
    So, there exists  $a \in \mathcal{A}$ such that $x\partial_{x}f^{2} - y\partial_{y}f^{2} + \varphi^{2} = h_{xy} a$, and we have:
    \begin{equation} \label{36}
     x\partial_{x}f^{1} + y\partial_{y}f^{1} - \varphi^{1} = a.
     \end{equation}
    It follows that $(i+j)f_{ij}^{1} - \varphi_{ij}^{1}= a_{ij}$. And then $\varphi_{ij}^{1} = (i+j)f_{ij}^{1} - a_{ij}$. Since we also have  $x\partial_{x}f^{2} - y\partial_{y}f^{2} + \varphi^{2} = h_{xy} a$, therefore we can see that $ \varphi^{2} = y\partial_{y}f^{2} - x\partial_{x}f^{2} - h_{xy} a$. So $\varphi_{ij}^{2} = (j-i)f_{ij}^{2} - h_{xy}a_{ij}$. By identifying members to members, $a_{ij}=0$. Then $\varphi^{1}$ and $\varphi^{2}$ are such that:
     \begin{eqnarray*}
               \left\{
                  \begin{array}{rl}
                 \varphi_{ij}^{2}= (j-i) f_{ij}^{2}~~~~~~~ \\
    \varphi_{ij}^{1}= (i+j) f_{ij}^{1}
                  \end{array}
                  \right.
                 \end{eqnarray*}
   We have  $Im d^{2}  \subseteq \underset{i,j\geq 0}{\bigoplus} (i+j)f_{ij}^{1}h_{xy}x^{i}y^{j}\mathbb{C} + \underset{i,j\geq 0}{\bigoplus} (j-i)x^{i}y^{j}\mathbb{C}$.
      And conversely, let us consider $f = (i+j)h_{xy}f_{ij}^{1}x^{i}y^{j} + (j-i)f_{ij}^{2}x^{i}y^{j}\in \underset{i,j\geq 0}{\bigoplus} (i+j)f_{ij}^{1}h_{xy}x^{i}y^{j}\mathbb{C} + \underset{i,j\geq 0}{\bigoplus} (j-i)x^{i}y^{j}\mathbb{C}$. We need to show that  $f \in Im d^{2}$.
One can verify that for all $g^{1}\delta^{1} + g^{2}\delta^{2}= \overset{\longrightarrow}{g}\in Der_{\mathcal{A}}(log \mathcal{I})$,  $d^{2}(g^{1}, g^{2}) = (i+j)g_{ij}^{1} h_{xy} x^{i}y^{j} + (j-i)g_{ij}^{2}x^{i}y^{j}$.
            By identification, just take $g_{ij}^{1}=f_{ij}^{1}$ and $g_{ij}^{2}=f_{ij}^{2}$ to justify the existence of $\overset{\longrightarrow}{g}$. Hence there exists an element $\overset{\longrightarrow}{g}\in Der_{\mathcal{A}}(log \mathcal{I})$ such that $d^{2}(g^{1}, g^{2}) = f \in Im d^{2}$. Then $\underset{i,j\geq 0}{\bigoplus} \left(i+j \right)h_{xy}x^{i}y^{j}\mathbb{C} + \underset{i,j\geq 0}{\bigoplus} (j-i)x^{i}y^{j}\mathbb{C} \subseteq Im d^{2}$.
                 It's follows that
                 $  Im d^{2} = \underset{i,j\geq 0}{\bigoplus} \left(i+j \right)h_{xy}x^{i}y^{j}\mathbb{C} + \underset{i,j\geq 0}{\bigoplus} (j-i)x^{i}y^{j}\mathbb{C}$.
It is well to remark that the submodule $\underset{i,j\geq 0}{\bigoplus} (j-i)x^{i}y^{j}\mathbb{C}$ does not contain the elements $f = f_{ii}x^{i}y^{i}$. Then we have
                 $Im d^{2} \simeq \underset{i,j\geq 0}{\bigoplus}h_{xy} x^{i}y^{j}\mathbb{C} +
                 \underset{\overset{i,j\geq 0}{i\neq j}}{\bigoplus} x^{i}y^{j}\mathbb{C}$. This completes the proof of the above lemma.
    \end{proof}

\begin{lemma}
      The first $\mathcal{A}$-module of    logarithmic Poisson cohomology is given by
        $H_{log}^{0}   \simeq \mathbb{C}$.
      \end{lemma}

    \begin{proof}
    $a\in Kerd ^{1}$ is equivalent to $d^{1}(a)=0$. Then $(\delta^{2}a; -\delta^{1}a) = (0;0)$. We can still say that
    $ \left(y\partial_{y}a- x\partial_{x}a;~~-h_{xy}( x\partial_{x}a + y\partial_{y}a ) \right) = (0;0)$. Which gives us $y\partial_{y}a- x\partial_{x}a=0$ and $ x\partial_{x}a + y\partial_{y}a=0$. By adding member  to member of these  equations we get $y\partial_{y}a = x\partial_{x}a=0$. It follows that  $a\in \mathbb{C}$. Hence $Kerd ^{1} = \mathbb{C}$. Since
  $Im d^{0}=0$, then $H_{log}^{0}  = \dfrac{Kerd^{1}}{Imd^{0}}  \simeq   \mathbb{C}$.
      \end{proof}
\begin{lemma} The second $\mathcal{A}$-module of  the  logarithmic Poisson cohomology is given by:\\
        $H_{log}^{1}  \simeq \mathbb{C}\times (\mathbb{C} \oplus xy \mathbb{C})$.
\end{lemma}
      \begin{proof}
      It's well known that $h_{xy}\mathbb{C}_{n}[x,y]$ is an ideal of $\mathbb{C}_{n}[x,y]$ generated by $\langle h_{xy}\rangle$. So $\dfrac{\mathbb{C}_{n}[x,y]}{\langle h_{xy}  \rangle} \simeq  r_{xy}\mathbb{C}$ where $r_{xy}$ are elements of  $\mathbb{C}_{n}[x,y]$ whose degree are strictly less than 2 in $x$.
      These elements are exactly the elements of the submodule $\mathbb{C} \oplus xy \mathbb{C}$.
      Therefore we have $\mathbb{C}_{n}[x,y] = ( \mathbb{C} \oplus xy \mathbb{C}) \oplus h_{xy}\mathbb{C}_{n}[x,y]$.
      Since  $Ker d^{2} \simeq  \mathbb{C} \times \mathbb{C}_{n}[x,y]$ and $Im d^{1}\simeq 0 \times h_{xy}\mathbb{C}_{n}[x,y]$, we can  deduce that
       $Ker d^{2} \backsimeq Im d^{1} \oplus \mathbb{C}\times (\mathbb{C} \oplus xy \mathbb{C})$.
 According to (the \textit{lemma} \ref{45}) we have
 $H_{log}^{1}   \simeq \mathbb{C}\times (\mathbb{C} \oplus xy \mathbb{C})$.
  This completes the proof.
     \end{proof}

 \begin{lemma} The third $\mathcal{A}$-module of  the  logarithmic Poisson cohomology is
        $H_{log}^{2}   \simeq \mathbb{C} \oplus xy \mathbb{C}$.
      \end{lemma}

    \begin{proof}
    Firstly, we have the 2-coboundary of inhomogeneous Poisson bracket  $\{-,-\}_{h}$ is given by
                  $Im d^{2} \simeq \underset{i,j\geq 0}{\bigoplus} h_{xy}x^{i}y^{j}\mathbb{C} +
                              \underset{\overset{i,j\geq 0}{i\neq j}}{\bigoplus} x^{i}y^{j}\mathbb{C}$.
    Secondly, the 3-cocycle of $\{-,-\}_{h}$ is $Kerd^{3} = \mathcal{A} = \underset{i,j\geq 0}{\bigoplus} x^{i}y^{j}\mathbb{C}$.
    The submodule of $\mathcal{A}$ which is not in $ \underset{i,j\geq 0}{\bigoplus} h_{xy}x^{i}y^{j}\mathbb{C}$ is
  $\mathbb{C}[x] \oplus \mathbb{C}[y] + xy \mathbb{C} = \mathbb{C} \oplus x\mathbb{C}[x] \oplus y\mathbb{C}[y]+ xy \mathbb{C}.$
    But $x\mathbb{C}[x] \oplus y\mathbb{C}[y] $ is the submodule of $ \underset{\overset{i,j\geq 0}{i\neq j}}{\bigoplus} x^{i}y^{j}\mathbb{C}$.
Therefore
     $ Kerd^{3} \simeq Im d^{2} \oplus  (\mathbb{C} \oplus xy \mathbb{C})$.  Then
     $H_{log}^{2}  \simeq \mathbb{C} \oplus xy \mathbb{C}$. This completes  the proof of the third logarithmic Poisson cohomology group given in the above lemma.
    \end{proof}

\begin{corollary}
We can remark that      $H_{log}^{2} $ is generated by the system $ \langle 1, xy  \rangle$, and we deduce the following relationships:
\begin{enumerate}
\item[1.] $H_{log}^{1}  \simeq \mathbb{C} \times H_{log}^{2} $,
\item[2.] $H_{log}^{2}  \simeq \mathbb{C}^{2}$ and $H_{log}^{1}  \simeq \mathbb{C}^{3}$,
\item[3.] $dim\left(H_{log}^{2}   \right) = 2$ and   $dim\left(H_{log}^{1}  \right) = 3$.
\end{enumerate}
\end{corollary}

\subsection{ Logarithmic De Rham Cohomology of $\{-,-\}_h.$}
Let $\mathcal{A}=\mathbb{C}[x]=\mathbb{C}[x_{1},...,x_{p}]$ and the free divisor $D=\{ h=0\}$, $h\in \mathcal{A}$. Noted by $A_{_{p}} = (x_{1}, ...,x_{p}; \partial_{1},..., \partial_{p})$ the Weyl algebra of order $p$ over the complex numbers $\mathbb{C}$ and by
 $(\Omega_{\mathcal{A}}^{\bullet}(log (h)),\tilde{d})$ or simply $(\Omega_{\mathcal{A}}^{\bullet}(logD),\tilde{d})$  the complex of polynomial where $\tilde{d}$ is the logarithmic exterior derivative.

      \begin{definition}
The logarithmic De Rham cochain complex is given by
  \[\xymatrix{0\ar[r]^{\tilde{d}}&\Omega^{0}_{\mathcal{A}}(logD)\ar[r]^{\tilde{d}}&\Omega^{1}_{\mathcal{A}}(logD)\ar[r]^{\tilde{d}}&... \ar[r]^{\tilde{d}}&\Omega^{p}_{\mathcal{A}}(logD)\ar[r]^{\tilde{d}}&0}\]
 where $\tilde{d}^{2}=0$. The associated cohomology is called the logarithmic De Rham cohomology.
      \end{definition}

We consider $f\in \mathcal{A}=\mathbb{C}[x,y]$ a non zero polynomial
such that $D=\{f=0\}$ and $A_{2}=(x,y,\partial x, \partial y)=A$ the
Weyl algebra of order 2. According to \cite{FJCJN}, the logarithmic
Spencer's cochain complex assorted to $M^{log(D)}$ for the fixed
Spencer basis $(\delta_{1}, \delta_{2})$ is given by:
\begin{equation}
 \xymatrix{ 0 \ar[r]^{\epsilon^{3}} & A \ar[r]^{\epsilon^{2}}  & A\times A
           \ar[r]^{\epsilon^{1}}  & A \ar[r]^{\epsilon^{0}}  & 0 }
\end{equation}
where $\epsilon ^{0}=\epsilon^{3}=0$ and $\epsilon ^{1};\epsilon^{2}$ are defined by
$\epsilon^{1}(a_{1},a_{2})= a_{1}\delta_{1} + a_{2} \delta_{2}$ for all $a_{i}\in A$, $\epsilon^{2}(g)= g( - \delta_{2}-b_{1}, \delta_{1}-b_{2})$ for all $g\in A$. Note that the polynomials $b_{i}$ verify the relation $[\delta_{1}, \delta_{2}] = \delta_{1} \delta_{2} - \delta_{2} \delta_{1} = b_{1}\delta_{1}+ b_{2}\delta_{2}$.
Applying the functor $Hom_{A}(-, \mathcal{A})$ to the above complex and using the natural isomorphism $\mathcal{A} = Hom_{A}(A, \mathcal{A})$ to obtain the De Rham cochain complex given as follows
\begin{equation} \label{11}
 \xymatrix{ 0 \ar[r]^{\epsilon^{*}_{0}} & \mathcal{A} \ar[r]^{\epsilon^{*}_{1}}  & \mathcal{A} \times \mathcal{A}
           \ar[r]^{\epsilon^{*}_{2}}  & \mathcal{A} \ar[r]^{\epsilon^{*}_{3}}  & 0 }
\end{equation}
where the cobord operators $\epsilon ^{*}_{k}$  are defined by   $\epsilon ^{*}_{0}=\epsilon^{*}_{3}=0$,
$\epsilon^{*}_{1}(a)= (\delta_{1}a, \delta_{2}a)$ and $\epsilon^{*}_{2}(a_{1},a_{2}) = \delta_{1}a_{2} - \delta_{2}a_{1}  -( b_{1}a_{1} + b_{2}a_{2})$. We consider de basis $(\omega_{1}, \omega_{2})$ dual of $(\delta_{1}, \delta_{2})$. We get the following commutative diagram:

 \[ \xymatrix{ 0\ar[r]^{\tilde{d}^0}& \mathcal{A} \ar[d]^{\tilde{\partial}^{0}}\ar[r] ^{\tilde{d}^1} &\Omega^{1}_{\mathcal{A}}(logD) \ar[d]^{\tilde{\partial}^1}\ar[r]^{\tilde{d}^2}&\Omega^{2}_{\mathcal{A}}(logD) \ar[d]^{\tilde{\partial}^2}\ar[r] ^{\tilde{d}^3}&0\\
     0\ar[r]^{\epsilon^{*}_{0}} & \mathcal{A} \ar[r]^{\epsilon^{*}_{1}}&\mathcal{A}^{2} \ar[r]^{\epsilon^{*}_{2}}& \mathcal{A} \ar[r]^{\epsilon^{*}_{3}}&0}
    \]
where $\tilde{\partial}^{0}$ is an identity, $\tilde{\partial}^{1} (a_{1}\omega_{1} +  a_{2}\omega_{2})  = (a_{1} , a_{2})$
 and $\tilde{\partial}^{2} (a\omega_{1} \wedge\omega_{2})  = a.$
 On the other words $\tilde{\partial}^{i+1}\circ \tilde{d}^{i+1} = \epsilon^{*}_{i+1}\circ  \tilde{\partial}^{i}$.\\
  In the following, we consider the fixed basis $(\delta^1, \delta^2)$ of $Der_{\mathcal{A}}(logD)$ given in  \textbf{lemma} \ref{31}.
  \begin{lemma} \label{L13}
The logarithmic De Rham's cochain complex associated to the inhomogeneous divisor $D = \{ h= xy+x^2y^2 + x^3y^3=0\}$ is
$ \xymatrix{ 0 \ar[r]^{\epsilon^{*}_{0}} & \mathcal{A} \ar[r]^{\epsilon^{*}_{1}}  & \mathcal{A} \times \mathcal{A}
           \ar[r]^{\epsilon^{*}_{2}}  & \mathcal{A} \ar[r]^{\epsilon^{*}_{3}}  & 0 }$
where the De Rham differentials are $\epsilon ^{*}_{0}=\epsilon^{*}_{3}=0$,
$\epsilon^{*}_{1}(a)= (\delta^1 a, \delta^2a)$ and $\epsilon^{*}_{2}(a_{1}, a_{2}) = \delta^1 a_{2} - \delta^2 a_{1}$.
  \end{lemma}
  \begin{proof}
According to the cochain complex given in $( \ref{11} )$, the De Rham differentials are given by $\epsilon^{*}_{1}(a)= (\delta^1 a, \delta^2a)$, $a\in \mathcal{A}$ and $\epsilon^{*}_{2}(a_{1},a_{2}) = \delta^1a_{2} - \delta^2a_{1}  -( b_{1}a_{1} + b_{2}a_{2})$, where $b_{i}$ verify the relation $[\delta^1, \delta^2]_{Der_{\mathcal{A}}(logD)}  = b_{1}\delta^1 + b_{2}\delta^2$.   See in \textbf{proposition} $\ref{P3}$ that $[\delta^1, \delta^2]_{Der_{\mathcal{A}}(logD)} =0$ since $[\delta^1, \delta^2]_{Der_{\mathcal{A}}(logD)} = \delta^1\circ \delta^2  - \delta^2\circ \delta^1$.  this implies that $b_{1} = b_{2}=0$. Furthermore $\epsilon^{*}_{2}\circ \epsilon^{*}_{1} = 0.$ This completes the prove of the above lemma.
  \end{proof}

\begin{corollary}
    The logarithmic Spencer's differential associate to $D = \{ h=0\}$ are  such that $\epsilon ^{0}=\epsilon^{3}=0$;
    $\epsilon^{1}(a_{1},a_{2})= a_{1}\delta^1 + a_{2} \delta^2$  and $\epsilon^{2}(g)= g( - \delta^2, \delta^1)$ for all $a_{i},g\in \mathcal{A}$.
    \end{corollary}
  In the following, we will compute the logarithmic De Rham cohomology  $H^{k}_{DR}$.
\begin{theorem}
 Let  $D= \{ h= xy+x^2y^2 + x^3y^3 = 0\}$ a free inhomogeneous divisor and $(\Omega_{\mathcal{A}}^{\bullet}(log(D)), \epsilon^{*}_{\bullet})$ the associate logarithmic  De Rham cochain complex. The logarithmic De Rham  cohomology are given as following:
 \begin{enumerate}
 \item[1.]  $H^{0}_{DR} \backsimeq   \mathbb{C}$,
 \item[2.]  $H^{1}_{DR}\simeq  0 \times \mathbb{C}$,
 \item[3.]  $H^{2}_{DR}\simeq \mathbb{C} \oplus xy \mathbb{C}$.
 \end{enumerate}
 \end{theorem}

 \begin{proposition}
  The  first logarithmic De Rham   cohomology group is $H^{0}_{DR} \backsimeq   \mathbb{C}$.
   \end{proposition}

 \begin{proof}
  For all $a\in \mathcal{A},$ $\epsilon^{*}_{1}(a)=0$ if and only if $a\in \mathbb{C}$.
   Its follows that $Ker(\epsilon^{*}_{1})=\mathbb{C}$. Since $Im(\epsilon^{*}_{0})= 0$, we have  $Ker(\epsilon^{*}_{1})= Im(\epsilon^{*}_{0})\oplus \mathbb{C} \simeq Im(\epsilon^{*}_{0}) \oplus H^{0}_{DR}$.
   \end{proof}

  \begin{proposition}
  The  second logarithmic De Rham  cohomology group is $H^{1}_{DR}\simeq  0 \times \mathbb{C}$,
    \end{proposition}

 \begin{proof}
 According to \textbf{lemma} \ref{L13}, we have $\epsilon^{*}_{1}(a)= (\delta^1 a, \delta^2a)$ and $\epsilon^{*}_{2}(a_{1}, a_{2}) = \delta^1 a_{2} - \delta^2 a_{1}$. It follows that,
  the $1$-cobord is $Im(\epsilon^{*}_{1})= h_{xy}\mathbb{C}_{n}[x,y]\times 0 $ and the $2$-cocycle is  given by $Ker(\epsilon^{*}_{2})=  h_{xy}\mathbb{C}_{n}[x,y]\times \mathbb{C}$. Therefore we get $Ker(\epsilon^{*}_{2})= Im(\epsilon^{*}_{1})\oplus( 0\times \mathbb{C})$. Then, we deduce the explicit expression of the  second logarithmic De Rham   cohomology group $H^{1}_{DR} \simeq  0 \times \mathbb{C}$.
    \end{proof}
 \begin{proposition}
  The  third logarithmic De Rham   cohomology group is  $H_{log}^{2}  \simeq \mathbb{C} \oplus xy \mathbb{C}$.
    \end{proposition}

\begin{proof}
 We use same method as in \textbf{lemma} $\ref{L9}$ to show that the $2$- cobord associated to the logarithmic De Rham  cochain complex is given by
                   $Im \epsilon^{2}_{*} \simeq \underset{i,j\geq 0}{\bigoplus}h_{xy} x^{i}y^{j}\mathbb{C} +
                              \underset{\overset{i,j\geq 0}{i\neq j}}{\bigoplus} x^{i}y^{j}\mathbb{C}$. Since the $3$-cocycle is
$Ker(\epsilon^{*}_{3})=  \mathcal{A}= Ker d^{3}$, hence we deduce that, the  third logarithmic De Rham  cohomology group and the third logarithmic Poisson cohomology group are isomorphics.
 Then $H^{2}_{DR} \simeq H_{log}^{2}\simeq \mathbb{C} \oplus xy \mathbb{C}$. This completes the proof.
  \end{proof}

Let us consider the inhomogeneous free divisor   $D= \{ h= 0\}$ with $h=\overset{m}{\underset{i=1}{\sum}}
\alpha_{i} x_{1}^{n_{i}} x_{2}^{n_{i}}$; $n_{i}>0$. Denote by $n_{min} = min(n_{1},n_{2},..., n_{m})$; $n_{max} = d^{\circ}(h)_{x_{i}}$ the degree of $h$ in
$x_{i}$ and $h = x_{1}^{n_{min}}x_{2}^{n_{min}}\tilde{h}(x_{1},x_{2})$.

  We consider the submodule of $\mathbb{C}_{n }[x_{1},x_{2}]$ noted by $\mathcal{F}$, the family of the inhomogeneous polynomials given as follows:
  \begin{eqnarray}
    \mathcal{F}= \left \{ f \in \left( \dfrac{\tilde{h}(x_{1},x_{2})-\alpha_{min}}{x_{1}x_{2}}\right) \mathbb{C} \subseteq \mathbb{C}_{n }[x_{1},x_{2}] \right \}  \label{13}
    \end{eqnarray}

\begin{lemma} Let $h=\overset{m}{\underset{i=1}{\sum}}
\alpha_{i} x_{1}^{n_{i}} x_{2}^{n_{i}} \in \mathcal{A}=\mathbb{C}[x_{1},x_{2}]$, the Euler vector field $E_{2} =
x_{1}\partial_{x_{1}} + x_{2} \partial_{x_{2}}$. $\label{16}$
\begin{equation}
  \mathcal{B}=( \delta^{1} =  \overset{m}{\underset{i=1}{\sum}}
  \alpha_{i} x_{1}^{n_{i}-1} x_{2}^{n_{i}-1}E_{2},~ \delta^{2} = x_{2} \partial_{x_{2}} -
x_{1}\partial_{x_{1}})
\end{equation}
and
\begin{equation}
 \mathcal{B}^*=(\omega_{1} = \dfrac{x_{2}}{2h}dx_{1} +
\dfrac{x_{1}}{2h}dx_{2},~ \omega_{2}= \dfrac{dx_{2}}{2x_{2}} - \dfrac{dx_{1}}{2x_{1}}).
\end{equation}
For $\mathcal{I}=h\mathcal{A}$, we have the following:
\begin{enumerate}
\item[1.] $\mathcal{B}$ is a Saito's basis of $Der_{\mathcal{A}}(log
\mathcal{I})$ and $\mathcal{B}^*$ is a Saito's basis of
$\Omega_{\mathcal{A}}(log \mathcal{I})$,
\item[2.] $\mathcal{B}$ and $\mathcal{B}^*$ are dual basis.
\end{enumerate}
\end{lemma}

  The complex of chain given in $( \ref{11} )$ and in the lemma $( \ref{43} )$ remain the same for this divisor.
Note that $\mathbb{C}\times \mathcal{F}$ is cohomological for
$H_{log}^{1}$  since $Imd^{1}  \cap (\mathbb{C} \times \mathcal{F})
= 0$ and $ \mathbb{C} \times \mathcal{F} \subset Ker d^{2}$. We also
remark that $\mathcal{F}$ is cohomological for $H^{2}_{log}$ since $
\mathcal{F} \cap Im d^{2}= 0 $ and  $ \mathcal{F} \subset Ker
d^{3}=\mathcal{A}$. Using the methods used for the computations of
previous cohomology, we get
 $H_{log}^{0}   \simeq  H^{0}_{DR} \simeq   \mathbb{C}$ and $H^{1}_{DR} \simeq  0 \times \mathbb{C}$.
But a more general result in $(\mathcal{A}= \mathbb{C}[x_{1},x_{2}], \{x_{1},x_{2}\}_{h}=h)$ is stated in the following theorem:
 \begin{theorem}  Let $D= \{ h=\overset{m}{\underset{i=1}{\sum}}
 \alpha_{i} x_{1}^{n_{i}} x_{2}^{n_{i}}=0; n_{i}>0\}$ an inhomogeneous  free divisor. Let $\tilde{h}_{f} \in \mathcal{F}$.  The logarithmic Poisson and the De Rham cohomologies are given as the following:
  \begin{enumerate}
  \item[1.]  $H_{log}^{1}  \simeq  \mathbb{C} \times  \mathbb{C}$ and $H_{log}^{2}  \simeq   \mathbb{C}  \simeq H^{2}_{DR}$, if $d^{\circ}(\tilde{h}_{f})_{x_{i}}= 0$.
  \item[2.] $H_{log}^{1}  \simeq  \mathbb{C} \times \left( \mathbb{C}\oplus \tilde{h}_{f}\mathbb{C} \right)$ and $H_{log}^{2}  \simeq   \mathbb{C}\oplus \tilde{h}_{f}\mathbb{C}   \simeq H^{2}_{DR}$, if $d^{\circ}(\tilde{h}_{f})_{x_{i}} > 0$.
  \end{enumerate}
 \end{theorem}

 \section{Conclusion}
We have met the dual challenge of investigating the properties of
inhomogeneous divisors. Previously underexposed, and of performing
explicit computation of logarithmic Poisson cohomology($H_{log}^{k}
$) and logarithmic De Rham cohomology($H^{k}_{DR}$). In our case in
point, it follows that, in dimension 2 the isomorphism between these
two cohomology groups do not holds( since $\pi$ is log-symplectic).

\bmhead{Acknowledgements}
 Sincerely, The authors wish to express their  gratitude to all active members
  of RTAG (the Research Team in Algebra and Geometry) at the University of Maroua
     for their significant contribution to this paper.  They also extend their thanks
      to the following authors, whose work made these results possible.

\section*{Declarations}
 We declare that the attached manuscript is original, has not been published previously, and is not currently under review by another journal.\\
\textbf{ Funding statement:} This research did not receive any specific grant from funding agencies in the public commercial, or not-for-profit sectors.\\

 \textbf{Conflict interest:} The authors have no relevant financial or non-financial interests to disclose.\\
\textbf{Author contributions: } The authors contributed to the study's conception and design, performed the research, and wrote the manuscript. All authors read and approved the final manuscript.\\

\textbf{Data availability statement:} The data that support the finding of this study are available in the references of this article via Digital Object Identifiers(DOI). Some references, however such as thesis work and publication from  older journal editions, do not have DOIs and are available from the corresponding author upon reasonable request.\\

\bibliography{sn-bibliography}

@article{JB,
  title={On generalized brackets and some of their applications},
  author={Braconnier, Jean},
  journal={Publications of the Department of Mathematics (Lyon)},
  volume={14},
  number={4},
  pages={77--106},
  year={1977}
}

@book{EF,
  title={Dissertation, normal crossing in local analytic geometry},
  author={Faber, Eleonore},
 month={august},
  year={2011},
  pages={1--46}
}

@inproceedings{FJCM,
  title={Logarithmic differential operators and logarithmic de Rham complexes relative to a free divisor},
  author={Calderon-Moreno, Francisco J},
  booktitle={Scientific Annals of the Higher Normal School},
  volume={32},
  number={5},
  pages={701--714},
  year={1999},
  organization={Elsevier}
}

@article{JH,
  author={Huebschmann, Johannes},
  title={Poisson cohomology and quantization},
  journal={Journal fur die reine und angewandte Mathematik},
  year={1990},
 volume={1990},
  number={408},
  pages={57--113},
  doi	= {10.1515/crll.1990.408.57}
}

@phdthesis{DJ1,
  title={Logarithmic Poisson Structures: Cohomological Invariants and Pre-Quantification},
  author={Dongho, Joseph},
  year={2012},
  school={University of Angers}
}

@article{DJ2,
  author={Joseph Dongho},
  title={Logarithmic Poisson cohomology for a simple normal crossing divisor},
  journal={Bulletin of the Iranian mathematical Society},
  year={2020},
  volum={46},
  number={5},
pages={1339--1357},  
 doi= {10.1007/s41980.-020-00390-w},
 publisher={Springer}
}

@article{FJCJN,
  title={The computation of the logarithmic cohomology for plane curves},
  author={Castro-Jimenez, Francisco Jesus and Takayama, Nobuki},
  journal={arXiv preprint arXiv:0712.0001},
     doi= {10.1016/j.jalgebra.2009.03.039},
  year={2007}
}

@article{AL,
  title={Manifold of Poisson and their associated Lie algebras},
  author={Lichnerowicz, Andre},
  journal={Journal of differential geometry},
  volume={12},
  number={2},
  pages={253--300},
  year={1977},
  publisher={Lehigh University}
}

@article{TM,
  title={Differential geometry, groups and Lie algebras, fibers and connections},
  author={Masson, Thierry},
  journal={Laboratory of Theoretical Physics, University Paris XI, Building},
  volume={210},
  number={91},
  pages={105},
  year={2001}
}

@article{PM,
  title={Poisson cohomology in dimension two},
  author={Monnier, Philippe},
  journal={Israel journal of mathematics},
  volume={129},
  number={1},
  pages={189--207},
  year={2002},
  doi= {10.1007/BF02773163},
  publisher={Springer}
}

@phdthesis{PA,
  title={Poisson (Co)homology and isolated singularities in small dimensions, with an application in deformation theory},
  author={Pichereau, Anne},
  year={2006},
  URL= {theses.fr/2006POIT2354},
  school={Poitiers}
}

@book{AJ,
  title={Invitation A Algebraic Topology Volume I: Homology},
  author={Jeanneret, Lines},
  volume={1},
  year={2014},
  publisher={Editions Cepadues}
}

@article{CRPV,
author={C. Roger and P. Vanhaecke},
title={Poisson cohomologyof the affine plane},
  institution={Girard Desargues Institute, University of Claude-Bernard (Lyon I)},
  adress={69622 Villeurbanne, France},
    doi	= {10.1006/jabr.2002.9147},
  year={2001}
}

@article{KS,
  author={Kyoji Saito},
  title={Theory of logarithmic differential forms and logarithmic vector fields},
  journal={Journal of the Faculty of Science, Univiversity of Tokyo Sect 1A, Mathematics},
  year={1980},
  volume={27},
  number={2},
  pages={265--291},
  doi= {10.15083/00039776}
}

@article{JS,
author={ Jiro Sekiguchi},
title={Three Dimensional Saito Free Divisors and Singular Curves},
  journal={Journal of Siberian Federal University . Mathematics \& Physics},
  year={2008},
    volume={1},
      number={1},
     pages={33--41},
      url={https:// www.mathnet.ru/jsfu4},
 note={cited by 6 articles},
  doi={10.48550/arXiv.2402.08305}
}

@article{IV,
author={I.Vaisman},
title={On the geometric quantization of poisson manifolds},
  journal={Journal of mathematical physics},
    volume={32},
   number={12},
  pages={3339--3345},
  year={1991},
  doi= {10.1063/1.529447}
}

@article{VK,
author={A. M. Vinogradov and I. S. Krasil'shchik},
title={What is the hamiltonian formalism, Russian mathematical surveys},
  journal={Journal of mathematical physics},
    volume={30},
   number={1},
  pages={177},
  doi= {10.1070/RM1975v030n01ABEH001403},
  year={1975}
}

\end{document}